\documentclass[12pt]{article}
\usepackage{amsmath}
\usepackage{amssymb}
\usepackage{amsfonts}
\usepackage{amsthm}
\usepackage{xcolor}
\usepackage{dashbox}
\usepackage{authblk}
\usepackage{mathrsfs}
\usepackage{tikz}
\usepackage{rotating}
\usepackage[all]{xy}
\newtheorem{theorem}{Theorem}
\newtheorem{lemma}{Lemma}
\newtheorem{corollary}{Corollary}
\newtheorem{proposition}{Proposition}
\theoremstyle{definition}
\newtheorem{definition}{Definition}
\newtheorem{example}{Example}

\newtheorem{problem}{Problem}
\theoremstyle{remark}

\numberwithin{equation}{section}

\providecommand{\keywords}[1]{{\noindent\textit{Keywords:}} #1}

\newcommand{\cfbox}[2]{%
    \colorlet{currentcolor}{.}%
    {\color{#1}%
    \fbox{\color{currentcolor}#2}}%
}

\begin{document}

\title{On super-strong Wilf equivalence classes of permutations}         

\author[1]{Demetris Hadjiloucas \thanks{d.hadjiloucas@euc.ac.cy}}
\author[1]{Ioannis Michos\thanks{i.michos@euc.ac.cy}}
\author[2]{Christina Savvidou \thanks{csavvid@ucy.ac.cy}} 
\affil[1]{Department of Computer Science and Engineering, School of Sciences, European University Cyprus, Cyprus}
\affil[2]{School of Computing and Mathematics, UCLan Cyprus, Cyprus}
\date{}

\maketitle

\begin{abstract}
Super-strong Wilf equivalence is a type of Wilf equivalence on words that was originally introduced as strong Wilf equivalence by Kitaev et al. [Electron. J. Combin. 16(2)] in $2009$. We provide a necessary and sufficient condition for two permutations in $n$ letters to be super-strongly Wilf equivalent, using distances between letters within a permutation. Furthermore, we give a characterization of such equivalence classes via two-colored binary trees. This allows us to prove, in the case of super-strong Wilf equivalence, the conjecture stated in the same article by Kitaev et al. that the cardinality of each Wilf equivalence class is a power of $2$.
\end{abstract}


\maketitle

\section{Introduction}

In this work we investigate the notion of \emph{super-strong Wilf equivalence} as given by J.~Pantone and V.~Vatter in \cite{Pantone} on permutations in $n$ letters. To avoid any confusion we note that this notion was originally referred to as \emph{strong Wilf equivalence} by S.~Kitaev et al.~in \cite{Kitaev}. Let $\mathbb{P}^*$ be the set of words on the alphabet $\mathbb{P}$ of positive integers. Following \cite{Pantone}, two words $u$ and $v$ are super-strongly Wilf equivalent (resp. strongly Wilf equivalent), denoted $u\sim_{ss} v$ (resp. $u\sim_{s} v$), if there exists a weight-preserving bijection $f: \mathbb{P}^* \rightarrow \mathbb{P}^*$ such that for all words $w$, the embedding sets of $u$ in $w$ and of $v$ in $f(w)$ are equal (resp. equipotent) (see Section \ref{Sec:Notation} for all relevant definitions). To our knowledge, only a limited number of results exist on super-strong Wilf equivalence. For example, even for $n=3$, it has been an open problem whether $213\sim_{ss} 312$ \cite[\S 8.4. Problem (6)]{Kitaev}. We show that the answer to this is affirmative (see Proposition \ref{prop:n->n-1}) and, moreover, we give a full characterization of super-strong Wilf equivalence classes. 

Our motivation arose mainly from another open problem on Wilf equivalence of permutations \cite[\S 8.4. Problem (5)]{Kitaev}, namely whether the number of elements of the symmetric group on $n$ letters that are Wilf equivalent to a given permutation is always a power of $2$. We are able to answer this positively in the case of super-strong Wilf equivalence.

A powerful tool for dealing with patterns in permutations is the \emph{cluster method} of Goulden and Jackson \cite{GouldenJackson, Kitaev, Elizalde}. Pantone and Vatter \cite{Pantone} used this method in the special case of embeddings in words. Our initial observation is based upon \cite[Theorem 1.1]{Pantone}, which states that two strongly Wilf equivalent words are rearrangements of one another. In Section \ref{Sec:MCRT} we extend this to a necessary and sufficient rearrangement criterion on minimal clusters (Minimal Cluster Rearrangement Theorem, MCRT for short) for super-strong Wilf equivalence (Theorem \ref{MCRT}), namely $u\sim_{ss} v$ if and only if every minimal cluster of $u$ is a rearrangement of the corresponding minimal cluster of $v$. An arithmetic interpretation of MCRT led us to an intersection rule result (Proposition \ref{intersectionrule}) by enumerating the number of times each letter is blocked by letters that are greater in an arbitrary pre-cluster.

The intersection rule implies preservation of distances under super-strong Wilf equivalence. This led us to define the notion of \emph{cross equivalence} (see Section \ref{Sec:SeqDiff}). We say that $u$ and $v$ are cross equivalent if for each letter $i$, $i^+(u) = i^+(v)$, where $i^+(u)$ denotes the multiset of distances of $i$ from letters greater than $i$ in $u$. 

In Section \ref{Sec:SeqDiff} we use the Inclusion-Exclusion Principle to simplify the intersection rule condition. This leads us to the notion of \emph{consecutive differences}. Given a permutation $u$ and a letter $i$, the vector of consecutive differences $\Delta_i(u^{-1})$ for $i\in [2,n-1]$, contains the distances between letters in $u$ that are greater than or equal to $i$ as they appear sequentially in $u$ from left to right. Our main result (Theorem \ref{seqofdiff}) is a concrete characterization of super-strong Wilf equivalence. In particular, $u\sim_{ss} v$ if and only if $u$ and $v$ have the same sequence of differences.

In Section \ref{Sec:BinaryTree} we define a binary tree $T^n(u)$ that helps us visualize the cross equivalence class of a given permutation $u$ as the set of leaves of $T^n(u)$. The crucial point here is that the cardinality of the latter is always a power of $2$. In order to partition this set into super-strong Wilf equivalence classes, we define a labeling on the vertices of $T^n(u)$ that have two children, distinguishing between ``good'' ones which preserve symmetry (labeled $0$), and ``bad'' ones which destroy symmetry (labeled $1$). This labeling, which is in accordance to the sequence of differences $\Delta_i(u^{-1})$ for $i\in [2,n-1]$, implies that the cardinality of each super-strong Wilf equivalence class is a power of $2$.

\section{Preliminaries}\label{Sec:Notation}

Let $\mathbb{P}$ be the set of positive integers with the usual total order $\leq$. For each positive integer $n$ we let $[n] = \{ 1, 2, \ldots , n \}$ and for two non-negative integers $m, n$, where $m < n$, we let
$[m, n] = \{ m, m+1, \ldots n \}$. Let $n\in \mathbb{P}$ and $S\subseteq \mathbb{P}$. We denote by $n+S$ the \emph{shift} of $S$ to the right by $n$ units, namely the set $n + S = \{ n+s: s\in S\}$. 

Let $\mathbb{P}^{*}$ be the {\em free monoid} on $\mathbb{P}$ with the operation of {\em concatenation} of words. The set $\mathbb{P}^{*}$ can also be viewed as the set of {\em strict integer compositions}. The set of words with letters from $[n]$, where each letter appears exactly once, is the set of permutations in $n$ letters, denoted by $\mathcal{S}_n$. Let $\epsilon$ be the empty word or composition. For every $w = w_{1}w_{2} \ldots w_{n} \in \mathbb{P}^*$, the \emph{reversal} $\tilde{w}$ of $w$ is defined as $\tilde{w} = w_n w_{n-1} \ldots w_2 w_1$. A word $w$ that is equal to its reversal is called a \emph{palindrome}. We let $|w|$ be the {\em length} $n$ of the word $w$ (i.e. the number of parts of the composition $w$) and $||w||$ be the {\em height} or {\em norm} of $w$ defined as $||w|| = w_{1} + w_{2} + \cdots + w_{n}$ (i.e. the total length of the composition $w$). We denote by $|w|_i$ the number of occurrences of the letter $i$ in $w$ and by $alph(w)$ the set of distinct letters of $\mathbb{P}$ that occur in $w$. Let us also define the multiset of distances between two distinct letters $i, j$ in $w \in \mathbb{P}^*$ as \[ d_w(i,j) = \{ |k - l| : w_k = i, w_l = j\}.\]
In the trivial case where $w$ is a permutation, $d_w(i,j)$ is a singleton, whose element is identified with the usual distance between the corresponding letters of the permutation. For example, for $w = 2132213$ we have $|w| = 7$, $||w|| = 14$, $|w|_2 = 3$, $alph(w) = \{1,2,3\}$, and $d_w(2,3) = \{2, 6, 1, 3, 2, 2\}$. \\


\noindent {\bf Generalized factor order} Given $w, u \in \mathbb{P}^{*}$, we say that $u$ is a {\em factor} of $w$ if there exist words $s, v \in \mathbb{P}^{*}$ such that $w = suv$. For example $u = 322$ is a factor of $w = 2132213$, since $w = 21u13$. Consider the poset $(\mathbb{P}, {\leq})$ with the usual order in $\mathbb{P}$. The {\em generalized factor order} on $\mathbb{P}^{*}$ is the partial order - also denoted by $\leq$ - obtained by letting $u \leq w$ if and only if there is a factor $v$ of $w$ such that $|u| = |v|$ and $u_{i} \leq v_{i}$,
for each $i \in [|u|]$.
The factor $v$ is called an {\em embedding} of $u$ in $w$. If the first element of $v$ is the $j$-th element of $w$ then the index $j$ is called
an {\em embedding index} of $u$ into $w$. The {\em embedding index set} of $u$ into $w$, or \emph{embedding set} for brevity, is defined as the set of all embedding indices of $u$ into $w$
and is denoted by $Em(u, w)$.

For example, if $u = 322$ and $w = 2343213421$, then $u \leq w$ with embedding factors $v = 343$, $v' = 432$ and $v'' = 342$ and corresponding embedding
index set $Em(u, w) = \{ 2, 3, 7 \}$. 

Let now $t, x$ be two commuting indeterminates. The {\em weight} of a word $w \in \mathbb{P}^{*}$ is defined as the monomial $wt(w) = t^{|w|}x^{||w||}$.
For example, for $w = 2132213$ we obtain $wt(w) = t^{7}x^{14}$. 

A bijection $f:\mathbb{P}^* \rightarrow \mathbb{P}^*$ is called {\em weight-preserving} if the weight of $w$ is preserved under $f$, i.e. $|f(w)| = |w|$ and $||f(w)|| = ||w||$, for every $w \in \mathbb{P}^{*}$. Observe that $f(\epsilon) = \epsilon$, for all weight-preserving bijections $f$. 

Let $u \in \mathbb{P}^{*}$. The {\em weight generating function} $F(u; t, x)$ of $u$ is defined in \cite{Kitaev} as
\[ F(u; t, x) = \sum_{w \geq u} wt(w) = \sum_{w \geq u} t^{|w|} x^{\|w\|}. \]
The generating function $A(u; t, x, y)$ of $u$ is defined in \cite{Pantone} as
\[ A(u; t, x, y) = \sum_{w\in \mathbb{P}^*} t^{|w|} x^{\|w\|} y^{|Em(u,w)|}. \]
There are three notions of Wilf equivalence that are relevant to this article. They are defined as follows. \\

\noindent {\bf Wilf equivalence} Two words $u, v \in \mathbb{P}^{*}$ are called {\em Wilf equivalent}, denoted $u \sim v$, if \[ F(u; t, x) = F(v; t, x). \] 

\noindent {\bf Strong Wilf equivalence} Two words $u, v \in \mathbb{P}^{*}$ are called {\em strongly Wilf equivalent}, denoted $u {\sim}_{s} v$, if \[ A(u; t, x, y) = A(v; t, x, y).\] Equivalently, $u\sim_s v$ if there exists a weight-preserving bijection $f : \mathbb{P}^{*} \rightarrow \mathbb{P}^{*}$ such that $|Em(u, w)| = |Em(v, f(w))|$ for all $w \in \mathbb{P}^{*}$.
\\

\noindent {\bf Super-strong Wilf equivalence} Two words $u, v \in \mathbb{P}^{*}$ are called {\em super-strongly Wilf equivalent}, denoted $u {\sim}_{ss} v$, if there exists a weight-preserving bijection $f : \mathbb{P}^{*} \rightarrow \mathbb{P}^{*}$ such that $Em(u, w) = Em(v, f(w))$ for all $w \in \mathbb{P}^{*}$.

We note that super-strong Wilf equivalence implies strong Wilf equivalence, which in turn implies Wilf equivalence. We denote by $[u]$ and $[u]_{ss}$ the Wilf and super-strong Wilf equivalence class respectively, of a given word $u$. \\

For a word $u=u_1 u_2 \ldots u_n$, define $u^+ = (u_1+1) (u_2+1) \ldots (u_n+1)$. The following result on Wilf and super-strong Wilf equivalences will be useful in the sequel.

\begin{lemma}\cite[Lemmas 4.1, 5.1]{Kitaev} \label{lem:Kitaev}
\begin{enumerate}
\item For every $u\in \mathbb{P}^*$, $u\sim \tilde{u}$.
\item If $u\sim v$, then (i) $1u\sim 1v$ and (ii) $u^+ \sim v^+$.
\item If $u\sim_{ss} v$, then (i) $1u\sim_{ss} 1v$, (ii) $1u \sim_{ss} v1$ and (iii) $u^+ \sim_{ss} v^+$.
\end{enumerate}
\end{lemma}

A well-known negative criterion for super-strong Wilf equivalence is related to the notion of \emph{minimal cluster} for a certain embedding index set, starting from position $1$. Such an embedding index set $E$ is completely characterized by the shift vector $(e_1, e_2, \ldots, e_r)$, which is defined by the equality
\begin{equation}\label{eq:EmbeddingSet}
E = \{ j_{0}, j_{1}, j_{2}, \ldots , j_{r} \} = \{ 1, 1+e_1, 1+e_1+e_2, \ldots, 1+e_1+e_2+\cdots + e_r\},
\end{equation} 
where $j_k = 1+e_1 + \cdots + e_k$, for $k\in [0,r]$.

Let $u$ be a word of length $n$ and $E$ be an embedding set, with the additional property that $1\leq e_i\leq n-1$. An $(r+1)$-\emph{pre-cluster} of $u$ with embedding set $E$, denoted $P(u,E)$, is an $(r+1) \times (e_1+e_2+\cdots+e_r+n)$ array where in the $i$-th row, there is a copy of the word $u$ shifted $e_1+\cdots+e_{i-1}$ places to the right and all remaining places - which are not included in a representation of a pre-cluster - are filled with $1$. An $(r+1)$-\emph{minimal cluster} $m(u,E)$ of $u$ with embedding set $E$ is the word of length $e_1+e_2+\cdots+e_r+n$ whose $j$-th letter is the maximum value in the $j$-th column of $P(u,E)$. 

Suppose that $u\in \mathcal{S}_n$ and $E$ is a given embedding set. Kitaev et al. used minimal clusters in \cite[p. 14]{Kitaev} to construct a word $w$ with $Em(u,w)=E$, such that $w$ has both minimum length and height. This can be done since the embedding set $Em(u, m(u,E))$ is uniquely defined by the positions of $n$ in $m(u,E)$. In the general case where $u$ is an arbitrary word $u\in \mathbb{P}^*$, $Em(u,m(u,E)) = E$ might not hold. For example, if $u=1\, 1\, 1$ and $E=\{1,3\}$ then it follows that $m(u,E) = 1\, 1\, 1\, 1\, 1$ and $Em(1\, 1\, 1, 1\, 1\, 1\, 1\, 1) = \{1,2,3\} \neq \{1,3\}$.



The above method yields a sufficiency criterion for non super-strong Wilf equivalence.

\begin{proposition}\label{Pr:NegSS}\cite[p.14]{Kitaev}
Let $u, v \in \mathcal{S}_n$ and let $E$ be an embedding index set. If $\|m(u,E)\| \neq \|m(v,E)\|$ then $u\nsim_{ss} v$.
\end{proposition}

\begin{proof}
Without loss of generality, we may assume that $||m(v,E)|| > ||m(u,E)||$. It is enough to show that for any weight preserving bijection $f : \mathbb{P}^{*} \rightarrow \mathbb{P}^{*}$ we get
$Em(u, m(u,E)) \neq Em(v, f(m(u,E)))$. Suppose the contrary. From the previous observation, $E=Em(u, m(u,E))$. Since $Em(v, f(m(u,E))) = E$, the minimality condition for $m(v,E)$ yields
$||f(m(u,E))|| \geq ||m(v,E)||$. By our hypothesis, we get $||f(m(u,E))|| > ||m(u,E)||$. But the latter contradicts our assumption that $f$ is height preserving.
\end{proof}


\section{Minimal Cluster Rearrangement Theorem}\label{Sec:MCRT}

The main result of this section is based upon results and tools from \cite{Pantone}. Borrowing notation from there, the \emph{minimal cluster generating function} of $u$ is defined to be
\[ M(u;t,x,z) = \sum_{r\geq 1} z^r \! \! \! \! \! \! \! \! \sum_{\text{minimal} \atop \text{$r$-clusters $m$ of $u$}} \! \! \! \! \! \! \! \! t^{|m|} x^{||m||}. \]

Two words $w$ and $w'$ in $\mathbb{P}^*$ are said to be \emph{rearrangements} of one another if $alph(w) = alph(w')$ and $|w|_i = |w'|_i$, for each $i\in alph(w)$. The main result in \cite{Pantone} is the following.

\begin{theorem}\cite[Theorem 1.1]{Pantone}
If two words in $\mathbb{P}^*$ are strongly Wilf equivalent then they are rearrangements of one another.
\end{theorem}
Minimal cluster generating functions are used to prove this; namely, it is shown that if $M(w;t,x,z) = M(w';t,x,z)$, then $w\sim_s w'$.

\vspace{0.2cm}

Suppose we have a word $u$. Let $v$ be the minimal cluster $v=m(u,E_1)$ of $u$, for some embedding set $E_1$. Consider also the minimal cluster $w=m(v,E_2)$ of $v$ with embedding set $E_2$. The \emph{Minimal Cluster Rearrangement Theorem} is based on a simple observation, namely that $w$ is also a minimal cluster of $u$. In particular, $w = m(u,E_3)$ with $E_3 = \{i+j-1: i\in E_1, j\in E_2\}$.

\begin{example}
Let $u = 2 3 1 4$ and $E_1 = \{1, 2, 4\}$. Constructing $P(u,E_1)$ and letting $v=m(u,E_1)$, we have:
\[ \begin{array}{ccccccccc}
         &   & 2 & 3 & 1 & 4 &   &   &   \\
         &   &   & 2 & 3 & 1 & 4 &   &   \\
         &   &   &   &   & 2 & 3 & 1 & 4 \\
\hline       v & = & 2 & 3 & 3 & 4 & 4 & 1 & 4. \\
     \end{array}
    \]
Suppose now that we take the minimal cluster $w=m(v,E_2)$ with $E_2 = \{1,3\}$, namely
\[ \begin{array}{ccccccccccc}
         &   & 2 & 3 & 3 & 4 & 4 & 1 & 4 &   &   \\
         &   &   &   & 2 & 3 & 3 & 4 & 4 & 1 & 4 \\
\hline       w & = & 2 & 3 & 3 & 4 & 4 & 4 & 4 & 1 & 4. \\
     \end{array}
    \]
According to the previous observation, $w$ is also a minimal cluster over $u$, with embedding set $E_3 = \{1+1-1, 1+3-1, 2+1-1, 2+3-1, 4+1-1, 4+3-1\} = \{1, 2, 3, 4, 6 \}$. To visualize this, substitute the words $v=m(u,E_1)$ in $P(v,E_2)$, with pre-cluster $P(u,E_1)$. This gives the following pre-cluster over $u$:
\[ \begin{array}{ccccccccccc}
         &   & 2 & 3 & 1 & 4 &   &   &   &   &   \\
         &   &   & 2 & 3 & 1 & 4 &   &   &   &   \\
         &   &   &   &   & 2 & 3 & 1 & 4 &   &   \\
         &   & - & - & - & - & - & - & - &   &   \\
         &   &   &   & 2 & 3 & 1 & 4 &   &   &   \\
         &   &   &   &   & 2 & 3 & 1 & 4 &   &   \\
         &   &   &   &   &   &   & 2 & 3 & 1 & 4 \\
         &   &   &   & - & - & - & - & - & - & - \\
\hline       w & = & 2 & 3 & 3 & 4 & 4 & 4 & 4 & 1 & 4. \\
     \end{array}
    \]
In this table, copies of the word $u$ start in positions $1\ (=1+1-1), 2\ (=1+2-1), 3\ (=3+1-1), 4\ (=1+4-1 \textrm{ or } 3+2-1)$ and $6\ (=3+4-1)$, as mentioned above.
\end{example}

The minimal cluster was defined over embedding sets $E$ (see (\ref{eq:EmbeddingSet})) under the restriction $e_i \in [n-1]$, to ensure that words in the pre-cluster always overlap. We need to extend this definition, so that the overlapping restriction is waived.

\begin{definition}
Let $m \in \mathbb{N}$, $u \in \mathcal{S}_n$ and $E = \{j_0, j_1, \ldots, j_r\}$ with $j_0=1$. The \emph{extended minimal cluster} of $u$ on $E$ with prescribed length $m$ is the unique word $w_{\min}$ of minimum height such that $Em(u, w_{\min}) = E$.
\end{definition}

It is obvious that an extended minimal cluster exists if and only if $m\geq j_r+|u| -1$. To obtain the extended minimal cluster of a word $u$, we construct once again its corresponding pre-cluster, filling any empty places with $1$.

\begin{example}
Suppose $m=11$, $u=2314$ and $E=\{1, 6, 7\}$. The corresponding extended minimal cluster is constructed as follows:

\[ \begin{array}{ccccccccccccc}
         &   & 2 & 3 & 1 & 4 &   &   &   &   &   &   & \\
         &   &   &   &   &   &   & 2 & 3 & 1 & 4 &   & \\
         &   &   &   &   &   &   &   & 2 & 3 & 1 & 4 & \\
\hline w & = & 2 & 3 & 1 & 4 & 1 & 2 & 3 & 3 & 4 & 4 & 1. \\
     \end{array}
    \]

\end{example}

\begin{theorem} \label{MCRT} {\bf (MCRT)}
Let $u_1,u_2 \in \mathbb{P}^*$. Then $u_1\sim_{ss} u_2$ if and only if the minimal clusters $m(u_1,E)$ and $m(u_2,E)$ are rearrangements of one another, for every embedding index set $E$.
\end{theorem}
\begin{proof}
Suppose that $u_1\sim_{ss} u_2$. Fix an embedding set $E$ that satisfies the overlapping condition. It is obvious that $|m(u_1, E)| = |m(u_2, E)|$ and by Proposition \ref{Pr:NegSS}, we get that $||m(u_1, E)|| = ||m(u_2, E)||$. Suppose $v_1 = m(u_1,E)$ and $v_2 = m(u_2, E)$.  We will show that $v_1 \sim_{s} v_2$. Following \cite[p.~4]{Pantone}, it suffices to show that $M(v_1;t,x,z) = M(v_2;t,x,z)$. From the previous discussion, and for any overlapping embedding set $E'$, we know that $m(v_1, E')$ and $m(v_2, E')$ are minimal clusters over $u_1$ and $u_2$ respectively, with embedding set $E'' = \{i+j-1: i\in E, j\in E'\}$. Since $u_1 \sim_{ss} u_2$, it follows that $m(v_1, E')$ and $m(v_2, E')$ have the same weight, hence $M(v_1;t,x,z) = M(v_2;t,x,z)$.

For the converse implication, suppose that for all embedding sets $E$, $m(u_1, E)$ and $m(u_2,E)$ are rearrangements of one another. We will construct a weight-preserving bijection $f$ from $\mathbb{P}^*$ to $\mathbb{P}^*$ such that ${Em}(u_1,w) = {Em}(u_2,f(w))$. First, we partition the set of words according to their length and height. Let $\mathbb{P}^*_{m,n} = \{w\in\mathbb{P}^*: |w|=n, ||w||=m\}$. Clearly, $\mathbb{P}^* = \{\epsilon\} \bigsqcup_{m,n\geq 1} \mathbb{P}^*_{m,n}$. To have a weight-preserving bijection $f$ from $\mathbb{P}^*$ to $\mathbb{P}^*$ such that ${Em}(u_1,w) = {Em}(u_2,f(w))$ it is necessary and sufficient to find a collection of bijections $f_{m,n}:\mathbb{P}^*_{m,n} \rightarrow \mathbb{P}^*_{m,n}$ such that ${Em}(u_1,w) = {Em}(u_2,f_{m,n}(w))$. Fix $m$ and $n$ in $\mathbb{N}$. Since $|\mathbb{P}^*_{m,n}| < \infty$, we know that we can find a bijection between two sets if and only if the two sets have the same cardinality. Note that since we have words of length $n$, we cannot place $u_1$ in position $n-u_1+2$ or anywhere beyond that. We show that for every possible embedding set $S\subseteq [n-|u_1|+1]$, the number of words $w$ in $\mathbb{P}^*_{m,n}$ such that ${Em}(u_1,w)=S$ is equal to the number of words $z$ in $\mathbb{P}^*_{m,n}$ such that ${Em}(u_2,z)=S$. This will imply the result. \\
Fix $S\subseteq [n-|u_1|+1]$. For every subset $T\subseteq [n-|u_1|+1]$, define $u_{1,T}$ to be the extended minimal cluster of $u_1$ with embedding set $T$. Define also $W_1(m, n, T) = \{w\in \mathbb{P}^*_{m,n}: {Em}(u_1,w) = T\}$ and $U_1(m, n, T) = \{w\in \mathbb{P}^*_{m,n}: T\subseteq {Em}(u_1,w)\}$. Clearly, $W_1\subseteq U_1$. The cardinality of $U_1(m,n,T)$ has an immediate combinatorial interpretation, via ordered weak partitions. To establish that the embedding indices of $u_1$ in a word $w$ contain those of $T$, it suffices to make sure that every letter of $w$ is greater than or equal to every letter of $u_{1,T}$ in the corresponding places. Therefore, such a $w$ can by constructed by partitioning the difference between heights $\|w\|=m$ and $\|u_{1,T}\|$ to the $n$ letters of $w$ (we illustrate this in an example below). To compute the cardinality of $W_1(m,n,S)$, we use the Inclusion-Exclusion principle and find that
\[ |W_1(m,n,S)| = \sum_{T\supseteq S} (-1)^{|T\setminus S|} |U_1(m,n,T)|.\]
If we replace $u_1$ by $u_2$, knowing that $||u_{1,T}|| = ||u_{2,T}||$ and $|u_{1,T}| = |u_{2,T}|$ implies that the cardinalities of the sets in the above equality remain the same. Thus, the desired equality $|W_1(m,n,S)| = |W_2(m,n,S)|$ follows.
\end{proof}


\begin{example}
Using the aforementioned notation, suppose that $m=9$, $n=4$, $u_1 = 2 3 1$, $T=\{1\}$, and $u_{1,T}=2 3 1 1$. Then $m-||u_{1,T}|| = 9-7 = 2$, and the $10$ ordered weak partitions of $2$ in 4 parts are
\[
(2,0,0,0), (0,2,0,0), (0,0,2,0), (0,0,0,2),\] \[ (1,1,0,0), (1,0,1,0), (1,0,0,1), (0,1,1,0), (0,1,0,1), (0,0,1,1).
\]
Thus, for $u_1 = 2 3 1$, by adding the vector $u_{1,T} = (2,3,1,1)$ to each one of the vectors above, we obtain
\[ U_1(9,4,\{1\}) = \{ 4 3 1 1, 2 5 1 1, 2 3 3 1, 2 3 1 3, 3 4 1 1, 3 3 2 1, 3 3 1 2, 2 4 2 1, 2 4 1 2, 2 3 2 2\}. \]
\end{example}

In view of the MCRT, in order to check whether two words $u, v \in {\mathcal {S}}_{n}$ are super-strongly Wilf equivalent, we must show that for an arbitrary fixed embedding set $E$, every letter in the minimal cluster of $u$ appears as many times as it appears in the corresponding minimal cluster of $v$. To do this we have to count the number of times each letter is inherited from a pre-cluster to a minimal cluster.

An embedding set $E$ can be written in the form (\ref{eq:EmbeddingSet}). Now for each $j \in [n]$ let ${\overline{j}}$ denote the shift of $E$ by $j - 1$ positions to the right, i.e.
\[ {\overline{j}} = (j - 1) + E = \{ j, j + e_{1}, j + e_{1} + e_{2}, \ldots , j + e_{1} + e_{2} + \cdots + e_{r} \} . \]

For a given $u = u_{1} u_{2} \cdots u_{i} \cdots u_{n} \in {\mathcal{S}}_{n}$ let $s = u^{-1} = s_{1} s_{2} \cdots s_{i} \cdots s_{n}$ denote its inverse in ${\mathcal{S}}_{n}$. For each $i \in [n]$ and a fixed embedding set $E$ consider the set $\overline{s_{i}}$. Clearly $s_{i}$ is the position of the letter $i$ in $u$, therefore the set $\overline{s_{i}}$ is precisely the set of all the positions of the letter $i$ in the pre-cluster of $u$ for the embedding $E$.

\begin{proposition} \label{intersectionrule}
Let $u, v \in {\mathcal{S}}_{n}$ and $s = u^{-1}, t = v^{-1}$. Then $u {\sim}_{ss} v$ if and only if
\begin{equation}\label{eq:InterSecRule}
 | \overline{s_{i}} \cap ( \bigcup_{j=i+1}^{n} \overline{s_{j}} ) | = | \overline{t_{i}} \cap ( \bigcup_{j=i+1}^{n} \overline{t_{j}} ) |, 
\end{equation}
for each $i \in [n-1]$ and every embedding set $E$.
\end{proposition}

\begin{proof}
Let $w$ and $w'$ be the minimal clusters of $u$ and $v$ with respect to a given embedding set $E$. In view of the MCRT, $u {\sim}_{ss} v$ is equivalent to the statement that the number of occurrences of each letter in $w$ is equal to the number of occurrences of the same letter in $w'$. The largest letter, namely $n$, appears the same number of times in $w$ and $w'$ as nothing can block it from being inherited. So the issue here is really about the letters in $[n-1]$. Let $i \in [n-1]$. Instead of counting the actual number of occurrences of $i$ in $w$ and $w'$, we do count the number of times that $i$ is blocked by bigger letters $j$ in the corresponding pre-clusters of $u$ and $v$ for an arbitrary fixed  embedding set $E$. This number is precisely
$| \overline{s_{i}} \cap ( \bigcup_{j=i+1}^{n} \overline{s_{j}} ) |$ for the word $u$ and
$| \overline{t_{i}} \cap ( \bigcup_{j=i+1}^{n} \overline{t_{j}} ) |$ for the word $v$ respectively and the result follows.
\end{proof}

\begin{proposition}\label{prop:n->n-1}
Let $n\in \mathbb{N}$ and $x, y, z \in \mathbb{P}^*$ such that $x(n-1)ynz \in \mathcal{S}_n$. Then \[ x(n-1)ynz \sim_{ss} xny(n-1)z.\]
\end{proposition}

\begin{proof}
Set $u=x(n-1)ynz$, $v=xny(n-1)z$, $s=u^{-1} = s_1 s_2 \ldots s_i \ldots s_{n-1} s_n$ and $t=v^{-1} = t_1 t_2 \ldots t_i \ldots t_{n-1} t_n$. Then we have $s_i = t_i$ for $i\in [n-2]$, $s_{n-1} = t_n = |x|+1$ and $s_n = t_{n-1} = |x| + |y| +2$. The equality in (\ref{eq:InterSecRule}) holds trivially for $i\in [n-2]$. For $i=n-1$ we have $|\overline{s_{n-1}} \cap \overline{s_n}| = |\overline{t_n} \cap \overline{t_{n-1}}| = |\overline{t_{n-1}} \cap \overline{t_n}|$ and the result follows.
\end{proof}

\noindent {\it Remark.} In the special case $n=3$, $y=1$ and $x=z=\epsilon$, the previous proposition gives an affirmative answer to conjecture \cite[\S 8.4, Problem (6)]{Kitaev}.

\section{Sequence of differences}\label{Sec:SeqDiff}
Proposition \ref{intersectionrule} implies that the distance between the positions of letters $n-1$ and $n$ is preserved under super-strong Wilf equivalence. It is natural to examine if this is the case for smaller letters too. 
Let $u \in {\mathcal{S}}_{n}$. For all $i \in [n-1]$ define $i^{+}(u)$ as the multiset of distances
\[ i^{+}(u) = \{ d_u(i, j) \, : \, j \in [i+1, n] \}. \]
An equivalent way to define $i^{+}(u)$ is via the inverse $s = s_{1} \cdots s_{i} \cdots s_{j} \cdots s_{n}$ of $u$, as
\[ i^{+}(u) = \{ |s_{i} - s_{j}| \, : \, j \in [i+1, n] \}. \]
Observe that any number in $i^+(u)$ appears at most two times. For example, let $n = 7$ and $u = 2361745$. Then $6^{+}(u) = \{ 2 \}$, $5^{+}(u) = \{ 2, 4 \}$, $4^{+}(u) = \{ 1, 1, 3 \}$,
$3^{+}(u) = \{ 1, 3, 4, 5 \}$, $2^{+}(u) = \{ 1, 2, 4, 5, 6 \}$ and finally $1^{+}(u) = \{ 1, 1, 2, 2, 3, 3 \}$.

\begin{definition}
Let $u, v \in {\mathcal{S}}_{n}$. We say that $u$ is \emph{cross equivalent} to $v$ and denote this by $u\sim_+ v$, if $i^{+}(u) = i^{+}(v)$, for all $i \in [n-1]$. 
\end{definition}
It is easy to check that cross equivalence is indeed an equivalence relation with $u\sim_+ \tilde{u}$. We denote by $[u]_+$ the cross equivalence class of the word $u$. We will show that it gives a necessary condition for super-strong Wilf equivalence.

\begin{proposition}
Let $u, v \in {\mathcal{S}}_{n}$. If $u {\sim}_{ss} v$ then $u {\sim}_{+} v$. The converse does not hold in general.
\end{proposition}

\begin{proof} 
Suppose that $u\nsim_+ v$ so that there exists an $i\in [n-1]$ such that $i^+(u) \neq i^+(v)$. We will show that $u\nsim_{ss} v$. Consider the set $D$ of all distances $d \in [n-1]$ such that 
\[ |\{d: d\in i^+(u)\}| \neq  |\{d: d\in i^+(v)\}|.\]
Let $e:= \min D$ and consider the embedding $E=\{1, 1+e\}$. We have the following two cases: $i$ appears in only one of the two multisets or $i$ appears in both of them, once and twice respectively.

Firstly, without loss of generality, $e\in i^+(u)\setminus i^+(v)$. Since $e\notin i^+(v)$, the letter $i$ will appear twice in the minimal cluster of $v$, since no letter greater than $i$ can block it. On the other hand, $i$ will be blocked at least once in the pre-cluster of $u$.

In the second case, without loss of generality, $e$ appears once in $i^+(u)$ and twice in $i^+(v)$. Then $i$ will be blocked twice in the pre-cluster of $v$, and exactly once in the pre-cluster of $u$.


Let us now construct a counterexample to show that the converse implication is not true. 
Set $u = 2351647$ and $v = 6471532$. It is easy to check that $u {\sim}_{+} v$. On the other hand, if we consider the embedding set $E = \{ 1, 2, 5 \}$
the corresponding minimal clusters for $u$ and $v$ are respectively $u_{min} = 235566776{\underline{4}}7$ and $v_{min} = 66776572532$. Clearly the letter $4$ appears only in $u_{min}$, so we immediately obtain $u {\nsim}_{ss} v$, by the MCRT.
\end{proof}

Proposition \ref{intersectionrule} gives a necessary and sufficient condition for super-strong Wilf equivalence. Nevertheless, it has not yet reached a concrete form involving the permutations in question. Using Inclusion - Exclusion Principle to simplify it, we are led to the following definition.

\begin{definition}\label{def:SeqDiff}
Let $u \in {\mathcal{S}}_{n}$ and $s = u^{-1}$. Let $s = s_{1} \cdots s_{i} \cdots s_{n}$. For $i = n - 1$ down to $1$ consider the
proper suffix $s_{i} \cdots s_{n}$ of $s$ and its alphabet set
${\Sigma}_{i}(s) = alph(s_{i} \cdots s_{n}) = \{ s_{i}^{(i)}, \ldots , s_{n}^{(i)} \}$,  where $s_{i}^{(i)} < \cdots < s_{n}^{(i)}$.
We define ${\Delta}_{i}(s)$ to be the vector of \emph{consecutive differences} in ${\Sigma}_{i}(s)$, i.e.
\[ {\Delta}_{i}(s) = ( s_{i+1}^{(i)} - s_{i}^{(i)}, \ldots , s_{n}^{(i)} - s_{n-1}^{(i)} ). \]
\end{definition}

As already mentioned in the Introduction, $\Delta_i(u^{-1})$ is the vector of distances between letters in $u$ that are greater than or equal to $i$ as they appear sequentially in $u$ from left to right. Note that since $\Delta_1(u^{-1})$ is always the $(n-1)$-tuple with all of its entries equal to $1$, we usually consider $\Delta_i(u^{-1})$ for $i = n-1$ down to $2$.

\begin{example}\label{ex:pyramiddiff}
Let $u = 21365874$. Then $s = u^{-1} = 21385476$. The sequence of differences for $s$ is the following:
\begin{center}
$\Delta_7(s) = (1)$ \\
$\Delta_6(s) = (2,1)$ \\
$\Delta_5(s) = (1,1,1)$ \\
$\Delta_4(s) = (1,1,1,1)$ \\
$\Delta_3(s) = (1,1,1,1,1)$ \\
$\Delta_2(s) = (2,1,1,1,1,1)$ \\
\end{center}
\end{example}

The main result of this section and the whole article is the following.

\begin{theorem}\label{seqofdiff}
Let $u, v \in {\mathcal{S}}_{n}$ and $s = u^{-1}, t = v^{-1}$. Then $u {\sim}_{ss} v$ if and only if ${\Delta}_{i}(s) = {\Delta}_{i}(t)$, for each
$i \in [2, n-1]$.
\end{theorem}

To prove this, we will need the following technical lemmas.

\begin{lemma} \label{differences}
Let $m \in [n]$ and $i_{1} < i_{2} < \ldots < i_{m}$, $j_{1} < j_{2} < \ldots < j_{m}$ be indices in $[n]$,
where $i_{1} \leq j_{1}$. Let $d_{l} = i_{l+1} - i_{l}$ and $f_{l} = j_{l+1} - j_{l}$, for $l \in [m-1]$, be respectively their consecutive differences.
Then the equality
\[ | {\overline{i_{1}}} \cap {\overline{i_{2}}} \cap \ldots \cap {\overline{i_{m}}}| =
   | {\overline{j_{1}}} \cap {\overline{j_{2}}} \cap \ldots \cap {\overline{j_{m}}}| \]
holds for every embedding set $E$ if and only if $(d_{1}, d_{2}, \ldots , d_{m-1}) = (f_{1}, f_{2}, \ldots , f_{m-1})$.
\end{lemma}

\begin{proof}
Suppose that $(d_{1}, d_{2}, \ldots , d_{m-1}) = (f_{1}, f_{2}, \ldots , f_{m-1})$. Then for each $k \in [m]$ we obtain $i_{k} = \sum_{p=1}^{k-1} d_{p} + i_{1}$ and $j_{k} = \sum_{p=1}^{k-1} f_{p} + j_{1}$, so that $j_{k} - i_{k} = j_{1} - i_{1}$. Let $d$ be this common difference of indices.

Now let $E$ be an embedding set and $X_{m} = {\overline{i_{1}}} \cap {\overline{i_{2}}} \cap \ldots \cap {\overline{i_{m}}}$ and
$Y_{m} = {\overline{j_{1}}} \cap {\overline{j_{2}}} \cap \ldots \cap {\overline{j_{m}}}$. We claim that the mapping $x \mapsto x + d$ is
a bijection from $X_{m}$ to $Y_{m}$. This is clearly a one-to-one mapping, so it suffices to show that for  each $x \in X_{m}$ we have $x + d \in Y_{m}$ and that for each $y \in Y_{m}$ we have $y - d \in X_{m}$. We show the former, the latter is left to the reader.

Let $(e_1, e_2, \ldots, e_r)$ be the vector that characterizes the embedding set $E$. For technical reasons we set $e_{0} = 0$. Since $x \in X_{m}$, there exists a strictly descending sequence of indices
${\alpha}_{1} > \cdots > {\alpha}_{l} > {\alpha}_{l+1} > \cdots > {\alpha}_{m}$, where ${\alpha}_{l} \in [0, r]$ and such that $x = i_{l} + e_{1} + \cdots + e_{{\alpha}_{l}}$, for each $l \in [m]$. This implies that $d_{l} = i_{l+1} - i_{l} = e_{{\alpha}_{l+1}+1} + \cdots + e_{{\alpha}_{l}}$, for each $l \in [m-1]$. Now since $d_{l} = f_{l}$, it
follows that $f_{l} = j_{l+1} - j_{l} = e_{{\alpha}_{l+1}+1} + \cdots + e_{{\alpha}_{l}}$, for each $l \in [m-1]$.
Then $x + d = x + (j_{1} - i_{1}) =  x + (j_{l} - i_{l}) = (i_{l} + e_{1} + \cdots + e_{{\alpha}_{l}}) + j_{l} - i_{l} = j_{l} + e_{1} + \cdots + e_{{\alpha}_{l}}$, for each $l \in [m]$, therefore $x + d \in Y_{m}$.

For the converse implication, suppose that $|X_{m}| = |Y_{m}|$, for every embedding set $E$. Consider, in particular the embedding set
\[ E = \{ 1, 1 + (i_{m} - i_{m-1}), 1 + (i_{m} - i_{m-2}), \ldots , 1 + (i_{m} - i_{1}) \}. \]
Then it is easy to see that $X_{m} = \{i_{m} \}$, therefore $|Y_{m}| = 1$. Now the only way that this can be done is when
\[ j_{m} = j_{m-1} + (i_{m} - i_{m-1}) = j_{m-2} + (i_{m} - i_{m-2}) = \cdots = j_{1} + (i_{m} - i_{1}). \]
The latter immediately implies that $j_{l+1} - j_{l} = i_{l+1} - i_{l}$, for each $l \in [m-1]$, as required.
\end{proof}

\begin{lemma} \label{comparison}
Let $u, v \in {\mathcal{S}}_{n}$ and let $s = s_{1} \cdots s_{i} \cdots s_{n}$, $t = t_{1} \cdots t_{i} \cdots t_{n}$ respectively be their inverses. Suppose that ${\Delta}_{i}(s) = {\Delta}_{i}(t)$ and ${\Delta}_{i+1}(s) = {\Delta}_{i+1}(t)$. 
\begin{enumerate}
\item If $s_{j}^{(i+1)} < s_{i} < s_{j+1}^{(i+1)}$, for some $j \in [i+1, n-1]$,
then $t_{j}^{(i+1)} < t_{i} < t_{j+1}^{(i+1)}$ with  $s_{i} - s_{j}^{(i+1)} = t_{i} - t_{j}^{(i+1)}$ and
$s_{j+1}^{(i+1)} - s_{i} = t_{j+1}^{(i+1)} - t_{i}$. 
\item If $s_{i} < s_{i+1}^{(i+1)}$ then 
\begin{enumerate}
\item either $t_{i} < t_{i+1}^{(i+1)}$ and $s_{i+1}^{(i+1)} - s_{i} = t_{i+1}^{(i+1)} - t_{i}$,
\item or $t_{i} > t_{n}^{(i+1)}$ and ${\Delta}_{i}(s) = {\Delta}_{i}(t) = (d, d, \ldots , d)$, where \[d = s_{i+1}^{(i+1)} - s_{i} = t_{i} - t_{n}^{(i+1)}.\]
\end{enumerate}
\item If $s_{i} > s_{n}^{(i+1)}$ then 
\begin{enumerate}
\item either $t_{i} > t_{n}^{(i+1)}$ and $s_{i} - s_{n}^{(i+1)} = t_{i} - t_{n}^{(i+1)}$,
\item or $t_{i} < t_{i+1}^{(i+1)}$ and ${\Delta}_{i}(s) = {\Delta}_{i}(t) = (d, d, \ldots , d)$, where \[d = s_{i} - s_{n}^{(i+1)} = t_{i+1}^{(i+1)} - t_{i}.\]
\end{enumerate}
\end{enumerate}
\end{lemma}

\begin{proof}
Suppose that $s_{j}^{(i+1)} < s_{i} < s_{j+1}^{(i+1)}$, for some $j \in [i+1, n-1]$. First we will show that $t_{i+1}^{(i+1)} < t_{i} < t_{n}^{(i+1)}$. Indeed, if $t_{i} < t_{i+1}^{(i+1)}$ then ${\Delta}_{i}(s) = {\Delta}_{i}(t)$ yields $s_{j+1}^{(i+1)} - s_{i} = t_{j+1}^{(i+1)} - t_{j}^{(i+1)}$. Since
${\Delta}_{i+1}(s) = {\Delta}_{i+1}(t)$, we get $t_{j+1}^{(i+1)} - t_{j}^{(i+1)} = s_{j+1}^{(i+1)} - s_{j}^{(i+1)}$. Thus we obtain
$s_{j+1}^{(i+1)} - s_{i} = s_{j+1}^{(i+1)} - s_{j}^{(i+1)}$, a contradiction. In a similar manner we cannot have $t_{i} > t_{n}^{(i+1)}$.
Therefore we necessarily get $t_{i+1}^{(i+1)} < t_{i} < t_{n}^{(i+1)}$.

Moreover, we will show that $t_{j}^{(i+1)} < t_{i} < t_{j+1}^{(i+1)}$. Suppose that $t_{j+k}^{(i+1)} < t_{i} < t_{j+k+1}^{(i+1)}$, for a suitable positive integer $k$. Then since ${\Delta}_{i}(s) = {\Delta}_{i}(t)$, we obtain $s_{i} - s_{j}^{(i+1)} = t_{j+1}^{(i+1)} - t_{j}^{(i+1)}$. Since the latter is equal to $s_{j+1}^{(i+1)} - s_{j}^{(i+1)}$ due to ${\Delta}_{i+1}(s) = {\Delta}_{i+1}(t)$, we obtain $s_{i} - s_{j}^{(i+1)} = s_{j+1}^{(i+1)} - s_{j}^{(i+1)}$, a contradiction. If $t_{j-k}^{(i+1)} < t_{i} < t_{j-k+1}^{(i+1)}$, for a suitable positive integer $k$, we interchange the role of $s$ and $t$ and work in a similar fashion. The equality $s_{j+1}^{(i+1)} - s_{i} = t_{j+1}^{(i+1)} - t_{i}$ follows from the assumption that ${\Delta}_{i}(s) = {\Delta}_{i}(t)$.

Now suppose that $s_{i} < s_{i+1}^{(i+1)}$. Then we show that $t_{i} \notin ( t_{i+1}^{(i+1)}, t_{n}^{(i+1)})$. Indeed, if the contrary holds, then by interchanging the roles of $s$ and $t$ we get that $s_{i} \in ( s_{i+1}^{(i+1)}, s_{n}^{(i+1)})$, which contradicts our assumption. Therefore, we either have
$t_{i} < t_{i+1}^{(i+1)}$ or $t_{i} > t_{n}^{(i+1)}$. In the former case the equality $s_{i+1}^{(i+1)} - s_{i} = t_{i+1}^{(i+1)} - t_{i}$ follows directly by the assumption that ${\Delta}_{i}(s) = {\Delta}_{i}(t)$. For the latter one we let $d = s_{i+1}^{(i+1)} - s_{i}$, $d' = t_{i} - t_{n}^{(i+1)}$ and
$d_{k} = s_{i+k+1}^{(i+1)} - s_{i+k}^{(i+1)} = t_{i+k+1}^{(i+1)} - t_{i+k}^{(i+1)}$, for $k \in [n-i-1]$. Since ${\Delta}_{i}(s) = {\Delta}_{i}(t)$, we finally obtain $d=d_{1}$, $d_{k} = d_{k+1}$, for $k \in [n-i-1]$ and $d_{n-i-1} = d'$. Thus all the consecutive differences are equal.

For the case where $s_{i} > s_{n}^{(i+1)}$, similar arguments as in the latter case apply.
\end{proof}

\begin{lemma}\label{palindrome}
Suppose that $u\sim_+ v$ and there exists an $i\in [2,n-2]$ such that $\Delta_{i}(s) \neq \Delta_{i}(t)$ and $\Delta_{i+1}(s) = \Delta_{i+1}(t)$. Then $\Delta_{i}(s) = \widetilde{\Delta_{i}(t)}$ and $\Delta_{i+1}(s)$ is a palindrome.
\end{lemma}

\begin{proof}
Suppose ${\Delta}_{i+1}(s) = {\Delta}_{i+1}(t) = (d_1, d_2, \ldots, d_{n-i-1})$ and $\Delta_i(s) \neq \Delta_i(t)$. Consider first the case where $i$ is placed, without loss of generality, to the left of all its greater letters in $u$ and therefore $\Delta_i(s) = (d_0, d_1, d_2, \ldots, d_{n-i-1})$. Then $\max (i^+(u)) = d_0 + d_1 + \cdots + d_{n-i-1}$. Since $i^+(u) = i^+(v)$, the same maximum is obtained only if $\Delta_i(t) =  (d_1, d_2, \ldots, d_{n-i-1}, d_0)$. It is immediate to see that $i^+(u) = \{ d_0 < d_0 + d_1 < d_0+d_1+d_2< \cdots < d_0 + d_1 + \cdots + d_{n-i-1} \}$ and $i^+(u) = \{ d_0 < d_0 + d_{n-i-1} < d_0 + d_{n-i-1} + d_{n-i-2} < \cdots < d_0 + d_{n-i-1} + \cdots + d_{1} \}$. Since $u\sim_+ v$, it follows that $\Delta_{i}(s) = \widetilde{\Delta_{i}(t)}$ and furthermore $\Delta_{i+1}(s)$ is a palindrome.

Now suppose that the letter $i$ appears in between larger letters at both $u$ and $v$. Then we have
\[ {\Delta}_{i}(s) = (d_1, d_2, \ldots, d_{k-1}, d'_k, d''_k, d_{k+1}, \ldots, d_{n-i-1}) \neq \] \[ {\Delta}_{i}(t) = (d_1, d_2, \ldots, d_{l-1}, e'_l, e''_l, d_{l+1}, \ldots, d_{n-i-1}), \] for suitable indices $k,l$. 
Set $i^+(u) = i^+(v) = M$. We distinguish between the following cases:

{\bf Case 1.} $k=l$: Then we claim that $n-i-1$ is odd, $k=(n-i)/2$, and $(d_1, d_2, \ldots, d_{n-i-1})$ is a palindrome. Since $i^+(u) = i^+(v) = M$, considering minimum elements, we obtain $\min\{d'_k, d''_k\} = \min\{e'_k, e_k''\}$. Clearly $d'_k \neq e'_k$, since ${\Delta}_{i}(s) \neq {\Delta}_{i}(t)$. Since $d'_k + d''_k = e'_k + e''_k = d_k$, it follows that $d'_k = e''_k$ and $d''_k = e'_k$. Going one step further for the multisets $M \setminus \{d'_k, d''_k\} = M \setminus \{e'_k, e''_k\}$, we obtain $\min\{d'_k+d_{k-1}, d''_k+d_{k+1}\} = \min\{e'_k+d_{k-1}, e_k''+d_{k+1}\}$, thus $d_{k-1} = d_{k+1}$. Repeating this process, we have $d_{k-j} = d_{k+j}$ for $j$ up to $r=\min\{k-1, n-i-k\}$. Suppose that $k-1\neq n-i-k$; without loss of generality $k-1<n-i-k$. Then $r=k-1$, $\max i^+(u) = d'_k + d_{k+1} + \cdots + d_{n-i-1}$, 
and $\max i^+(v) = d''_k + d_{k+1} + \cdots + d_{n-i-1}$, which leads to a contradiction.

{\bf Case 2.} $k<l$: Considering the maximum element of $i^+(u)$ and $i^+(v)$, we obtain that 
$\max\{d_k'+d_{k-1}+\cdots+d_1, d_k'' + d_{k+1} + \cdots + d_{n-i-1}\} = \max\{e_l'+d_{l-1}+\cdots+d_1, e_l'' + d_{l+1} + \cdots + d_{n-i-1}\}$. It follows that $d_k'+d_{k-1}+\cdots+d_1 = e_l'' + d_{l+1} + \cdots + d_{n-i-1}$ and $e_l'+d_{l-1}+\cdots+d_1 = d_k'' + d_{k+1} + \cdots + d_{n-i-1}$. Deleting these two elements from $i^+(u)$ and $i^+(v)$, respectively, we consider the two new possible choices for maximum and we get 
$d_k'+d_{k-1}+\cdots+d_2 = e_l'' + d_{l+1} + \cdots + d_{n-i-2}$ and $e_l'+d_{l-1}+\cdots+d_2 = d_k'' + d_{k+1} + \cdots + d_{n-i-2}$. Thus, $d_1 = d_{n-i-1}$. 
Repeating this process, we have $d_j = d_{n-i-j}$ for $j\in [r]$, where $r=\min\{k, n-i-1-l\}$. We claim that $k=n-i-1-l$. Suppose for the sake of contradiction, without loss of generality, that $k< n-i-1-l$. After $k$ successive deletions of the distances from the leftmost and rightmost elements, we obtain a common multiset $M_k$. If we compute the maximum of $M_k$ with respect to $u$ we get $\max M_k = d''_k + d_{k+1} + d_{k+2} + \cdots + d_{n-i-1-k}$, whereas doing the same with respect to $v$ yields $\max M_k < d_{k+1} + d_{k+2} + \cdots + d_{n-i-1-k}$, which clearly cannot hold. 

Since $k=n-i-1-l$, on the one hand with respect to $u$, we have $M_k = \{d''_k < d''_k + d_{k+1}< \cdots < d''_k + d_{k+1} + \cdots + d_{l-1} \}$ and on the other with respect to $v$, we obtain $M_k = \{e'_l < e'_l + d_{l-1}< \cdots < e'_l + d_{l-1} + \cdots + d_{k+1} \}$. Then it is easy to conclude that $d''_k = e'_l, d_{k+1} = d_{l-1}, d_{k+2} = d_{l-2},$ and so on.
\end{proof}

\vspace{0.5cm}

\begin{proof}[Proof of theorem \ref{seqofdiff}] \ \\
\emph{The condition ${\Delta}_{i}(s) = {\Delta}_{i}(t)$, for each $i \in [2, n-1]$ is sufficient}: Suppose that ${\Delta}_{i}(s) = {\Delta}_{i}(t)$, for each $i \in [2, n-1]$. We will show that $u {\sim}_{ss} v$ using Proposition \ref{intersectionrule}. Using previous notation and the Inclusion-Exclusion Principle we get

\[ | \overline{s_{i}} \cap ( \bigcup_{j=i+1}^{n} \overline{s_{j}} ) | = | \overline{s_{i}} \cap ( \bigcup_{j=i+1}^{n} \overline{s_{j}^{(i+1)}} ) | =
| \bigcup_{j = 1}^{n-i} ( \overline{s_{i}} \cap \overline{s_{i+j}^{(i+1)}} ) | \]

\begin{equation}\label{IEP1}
 = \sum_{k=1}^{n-i-1} {(-1)}^{k+1} \!\!\!\!\!\!\!\!\!\!\!\!\!\sum_{1 \leq j_{1} < \cdots < j_{k} \leq n-i} | \overline{s_{i}} \cap \overline{s_{i + j_{1}}^{(i+1)}}  \cap \cdots \cap \overline{s_{i+j_{k}}^{(i+1)}} |  + {(-1)}^{n-i+1} | \overline{s_{i}} \cap \overline{s_{i+1}^{(i+1)}} \cap \cdots \cap \overline{s_{n}^{(i+1)}} |. \end{equation}
Similarly we obtain
\[ | \overline{t_{i}} \cap ( \bigcup_{j=i+1}^{n} \overline{t_{j}} ) | = \]
\begin{equation}\label{IEP2} 
= \sum_{k=1}^{n-i-1} {(-1)}^{k+1} \!\!\!\!\!\!\!\!\!\!\!\!\!\sum_{1 \leq j_{1} < \cdots < j_{k} \leq n-i} | \overline{t_{i}} \cap \overline{t_{i + j_{1}}^{(i+1)}}  \cap \cdots \cap \overline{t_{i+j_{k}}^{(i+1)}} |  + {(-1)}^{n-i+1} | \overline{t_{i}} \cap \overline{t_{i+1}^{(i+1)}} \cap \cdots \cap \overline{t_{n}^{(i+1)}} |. \end{equation}

In view of Lemma \ref{comparison}, we distinguish between $3$ cases:

{\bf $(1)$} {\em $s_{i} < s_{i+1}^{(i+1)}$ and $t_{i} < t_{i+1}^{(i+1)}$}. \\
Since ${\Delta}_{i}(s) = {\Delta}_{i}(t)$, Lemma \ref{differences} immediately yields the equality
\[ | \overline{s_{i}} \cap \overline{s_{i+1}^{(i+1)}} \cap \cdots \cap \overline{s_{n}^{(i+1)}} | =
   | \overline{t_{i}} \cap \overline{t_{i+1}^{(i+1)}} \cap \cdots \cap \overline{t_{n}^{(i+1)}} | \]
between the last terms in (\ref{IEP1}) and (\ref{IEP2}).

Furthermore, ${\Delta}_{i}(s) = {\Delta}_{i}(t)$ implies the following equality of coarser differences
\[ (s_{i+j_{1}}^{(i+1)} - s_{i}, \ldots , s_{i+j_{k}}^{(i+1)} - s_{i+j_{k-1}}^{(i+1)}) =
   (t_{i+j_{1}}^{(i+1)} - t_{i}, \ldots , t_{i+j_{k}}^{(i+1)} - t_{i+j_{k-1}}^{(i+1)}). \]
Then Lemma \ref{differences} once more implies that
\[ | \overline{s_{i}} \cap \overline{s_{i + j_{1}}^{(i+1)}}  \cap \cdots \cap \overline{s_{i+j_{k}}^{(i+1)}} | =
   | \overline{t_{i}} \cap \overline{t_{i + j_{1}}^{(i+1)}}  \cap \cdots \cap \overline{t_{i+j_{k}}^{(i+1)}} |, \]
so that every term for $s$ in (\ref{IEP1}) is equal to the corresponding one for $t$ in (\ref{IEP2}).

The dual case $s_{i} > s_{n}^{(i+1)}$ and $t_{i} > t_{n}^{(i+1)}$ is dealt in a similar way.

\vspace{0.3cm}

{\bf $(2)$} {\em $s_{i+l}^{(i+1)} < s_{i} < s_{i+l+1}^{(i+1)}$, for some $l \in [n-1+i]$}. \\
By Lemma \ref{comparison} we immediately get $t_{i+l}^{(i+1)} < t_{i} < t_{i+l+1}^{(i+1)}$, for the same index $l$.
Rearranging terms we obtain
\[ | \overline{s_{i}} \cap \overline{s_{i+1}^{(i+1)}} \cap \cdots \cap \overline{s_{n}^{(i+1)}} | =
   | \overline{s_{i+1}^{(i+1)}} \cap \ldots \cap \overline{s_{i+l}^{(i+1)}} \cap \overline{s_{i}} \cap \overline{s_{i+l+1}^{(i+1)}}  \cap \cdots \cap \overline{s_{n}^{(i+1)}} |. \]
Since ${\Delta}_{i}(s) = {\Delta}_{i}(t)$, Lemma \ref{differences} implies that the latter term is equal to
\[| \overline{t_{i+1}^{(i+1)}} \cap \ldots \cap \overline{t_{i+l}^{(i+1)}} \cap \overline{t_{i}} \cap \overline{t_{i+l+1}^{(i+1)}} \cap \cdots \cap
   \overline{t_{n}^{(i+1)}} |,\]  which is clearly identical to
$| \overline{t_{i}} \cap \overline{t_{i+1}^{(i+1)}} \cap \cdots \cap \overline{t_{n}^{(i+1)}} |$.

A similar rearrangement of terms would lead us to compare the cardinalities

\[
| \overline{s_{j_i+1}^{(i+1)}} \cap \ldots \cap \overline{s_{i+j_m}^{(i+1)}} \cap \overline{s_{i}} \cap \overline{s_{i+j_{m+1}}^{(i+1)}}  \cap \cdots \cap \overline{s_{i+j_k}^{(i+1)}} |
\]
and
\[
| \overline{t_{j_i+1}^{(i+1)}} \cap \ldots \cap \overline{t_{i+j_m}^{(i+1)}} \cap \overline{t_{i}} \cap \overline{t_{i+j_{m+1}}^{(i+1)}}  \cap \cdots \cap \overline{t_{i+j_k}^{(i+1)}} |,
\]
for a suitable index $m$. Once more, $\Delta_i(s) = \Delta_i(t)$ implies the following equality of coarser differences

\[ (s_{i+j_2}^{(i+1)} - s_{i+j_1}^{(i+1)}, \ldots, s_i - s_{i+j_m}^{(i+1)}, s_{i+j_{m+1}}^{(i+1)} - s_{i}, \ldots , s_{i+j_{k}}^{(i+1)} - s_{i+j_{k-1}}^{(i+1)}) = \] \[
   (t_{i+j_2}^{(i+1)} - t_{i+j_1}^{(i+1)}, \ldots, t_i - t_{i+j_m}^{(i+1)}, t_{i+j_{m+1}}^{(i+1)} - t_{i}, \ldots , t_{i+j_{k}}^{(i+1)} - t_{i+j_{k-1}}^{(i+1)}). \]
Now the result follows immediately by Lemma \ref{differences}.

\vspace{0.3cm}

{\bf $(3)$} {\em $s_{i} < s_{i+1}^{(i+1)}$ and $t_{i} > t_{n}^{(i+1)}$}. \\
By a direct application of Lemma \ref{comparison} (Case 2(b)), we obtain that the consecutive differences in both $| \overline{s_{i}} \cap \overline{s_{i+1}^{(i+1)}} \cap \cdots \cap \overline{s_{n}^{(i+1)}} |$ and  $| \overline{t_{i}} \cap \overline{t_{i+1}^{(i+1)}} \cap \cdots \cap \overline{t_{n}^{(i+1)}} |$ is $(d,d,\ldots,d)$. Therefore, by Lemma \ref{differences}, we obtain that 
$| \overline{s_{i}} \cap \overline{s_{i+1}^{(i+1)}} \cap \cdots \cap \overline{s_{n}^{(i+1)}} | = | \overline{t_{i}} \cap \overline{t_{i+1}^{(i+1)}} \cap \cdots \cap \overline{t_{n}^{(i+1)}} |$.

For the previous terms of the summations in Equations (\ref{IEP1}) and (\ref{IEP2}), it suffices to construct a bijection from the set $\{1\leq j_1<\ldots < j_k\leq n-i\}$ to itself, that will preserve the equality of the corresponding sums there. This is equivalent to constructing a bijection $\phi$ from the set $\{(j_0, j_1, \ldots, j_k) : j_0=0  < j_{1} < \cdots < j_{k} \leq n-i\}$ to $\{(j_1,\ldots,j_k,j_{k+1} ) : 1\leq  j_{1} < \cdots < j_{k} < j_{k+1} = n+1-i\}$, that will preserve the equality
\begin{equation}\label{eq:translation}
 | \overline{s_{i}} \cap \overline{s_{i + j_{1}}^{(i+1)}}  \cap \cdots \cap \overline{s_{i+j_{k}}^{(i+1)}} | 
= 
|  \overline{t_{i + \phi_1(\alpha)}^{(i+1)}}  \cap \cdots \cap \overline{t_{i+\phi_k(\alpha)}^{(i+1)}} \cap \overline{t^{(i+1)}_{n+1}}|,
\end{equation}
where by convention $\overline{t_i}:= \overline{t^{(i+1)}_{n+1}}$, $\alpha=(0,j_1, \ldots, j_k)$, $n_i=n+1-i$ and the bijection $\phi$ is defined via its coordinate functions as
\[ \phi(\alpha) = (\phi_1(\alpha), \phi_2(\alpha), \ldots, \phi_k(\alpha), \phi_{k+1}(\alpha))
=(n_i-j_k, n_i-j_k + j_1, \ldots, n_i-j_k + j_{k-1}, n_i). \]

By Lemma \ref{comparison} (Case 2(b)), we have that $s^{(i+1)}_{i+j_{l}} - s^{(i+1)}_{i+j_{l-1}} = d(j_{l}-j_{l-1})$ and $t^{(i+1)}_{i+\phi_{l+1}(\alpha)} - t^{(i+1)}_{i+\phi_l(\alpha)} = d(\phi_{l+1}(\alpha)-\phi_l(\alpha))$. Now, by a careful analysis of the definition of $\phi$, it follows that in every case $\phi_{l+1}(\alpha)-\phi_l(\alpha) = j_{l} - j_{l-1}$, for $l=1,\ldots, k$. The equality (\ref{eq:translation}) now follows by Lemma \ref{differences}.

The dual case $s_{i} > s_{n}^{(i+1)}$ and $t_{i} < t_{i+1}^{(i+1)}$ is dealt in a similar way.

\vspace{0.5cm}

\emph{The condition ${\Delta}_{i}(s) = {\Delta}_{i}(t)$, for each $i \in [2, n-1]$ is necessary}: Suppose that $u\sim_{ss} v$. We will show that ${\Delta}_{i}(s) = {\Delta}_{i}(t)$, for each $i\in [2,n-1]$. Suppose the contrary. Let $i$ be the largest index in $[2,n-1]$ such that ${\Delta}_{i}(s) \neq {\Delta}_{i}(t)$. If $i=n-1$, then by Lemma \ref{differences} it follows that there exists an embedding $E$ such that $|\overline{s_n^{(n-1)}}\cap \overline{s_{n-1}^{(n-1)}}|\neq |\overline{t_n^{(n-1)}}\cap \overline{t_{n-1}^{(n-1)}}|$. Therefore, by Proposition \ref{intersectionrule}, $u\nsim_{ss} v$, a contradiction. Thus we may assume that $i<n-1$. 

By Lemma \ref{palindrome}, we know that $\Delta_{i+1} = \Delta_{i+1}(s) = \Delta_{i+1}(t) = (d_1, d_2, \ldots, d_{n-i-2}, d_{n-i-1})$ is a palindrome. Thus, $d_k = d_{n-i-k}$ for all $1\leq k \leq \lceil \frac{n-i-1}{2}\rceil$. Therefore, the factors of the words $u$ and $v$ that correspond to the previous distance vector $\Delta_{i+1}$ may be written in the form
\begin{equation}\label{configuration} 
* \ \underbrace{\circ \cdots \circ}_{d_1-1} \ * \ \underbrace{\circ \cdots \circ}_{d_2-1} \ * \ \cdots \ * \ \underbrace{\circ \cdots \circ}_{d_2-1} \ * \ \underbrace{\circ \cdots \circ}_{d_1-1} *,
\end{equation}
where $*$ corresponds to letters greater than $i$ and $\circ$ corresponds to letters less than or equal to $i$. The crucial point is the placement of the letter $i$ on $u$ and $v$. 

First, we consider the case where $i$ is placed in between greater letters. It will replace one of the characters $\circ$ in the above configuration, in distinct positions for $u$ and $v$ respectively, since $\Delta_{i}(s) \neq \Delta_{i}(t)$. Let $r$ be the distance of the letter $i$ from the leftmost (respectively, rightmost) greater letter. Since $\Delta_{i}(s) = \widetilde{\Delta_{i}(t)}$, after the insertion of the letter $i$, without loss of generality, we have the following configurations of common length $m$
\[ * \ \underbrace{\cdots}_r \ \fbox{i} \ \cdots \ \circ \ \underbrace{\cdots}_r \ * \quad \textrm{ and } \quad 
* \ \underbrace{\cdots}_r \ \circ \ \cdots \ \fbox{i} \ \underbrace{\cdots}_r \ *, 
\]
for the corresponding factors of $u$ and $v$, respectively. We want to count the number of times that the letter $i$ is inherited in some minimal cluster for $u$ and $v$, hence the factors in $u$ and $v$ on the left and on the right of the above configurations contain only letters smaller than $i$ and they do not affect us.

These may be written in a more precise form as follows:
\[ 
* \ u_1 \ \fbox{i} \ v \ \circ \ u_2 \ * \quad \textrm{ and } \quad * \ u'_2 \ \circ \ v' \ \fbox{i} \ u'_1 \ * ,
\]
where $|u_1| = |u_2| = |u_1'| = |u_2'| = r$ and $|v| = |v'|$.

We distinguish between two cases, $r < |v|$ and $r\geq |v|$. In the former case, $v$ and $v'$ can be respectively written as $v = u_3 b w$ and $v' = w' b' u'_3$, where $|u_3|=|u_3'|=r$; $b, b'$ are letters, and $w, w'\in \mathbb{P}^*$. Consider the embedding $E=\{1, r+2, m\}$. We have the following parts of the pre-clusters for $u$ and $v$, respectively

\[ \begin{array}{ccccccccccccccccc}
       * & u_1 & \fbox{${\bf{i}}$} & u_3 & b & w   & \circ & u_2 & * & & & &  & & & & \\
         &  & * & u_1 & \fbox{${\bf{i}}$}    & \cdots & \cdots & \cdots & \circ & u_2 & * & & & & & & \\
      & &  &  &  &  &  & & * & u_1 & \fbox{${\bf{i}}$} & u_3 & b & w   & \circ & u_2 & *  \\
\end{array} \]

\vspace{1cm}

\[ \begin{array}{ccccccccccccccccc}
       * & u'_2 & \circ & w' & b' & u_3'   & \fbox{${\bf{i}}$} & u'_1 & * & & & &  & & & & \\
         &  & * & \cdots & \cdots    & \cdots & b' & u_3' & \fbox{${\bf{i}}$} & u'_1 & * & & & & & & \\
      & &  &  &  &  &  & & * & u'_2 & \circ & w' & b' & u_3'   & \fbox{${\bf{i}}$} & u'_1 & *  \\
\end{array} \]
We claim that $b>i$ if and only if $b'>i$. In the notation of the proof of Lemma \ref{palindrome}, we observe that $b > i$ if and only if $r = d''_k + d_{k+1} + \cdots + d_{k+q}$, for a suitable $q\geq 0$. Since $d''_k = e'_l, d_{k+1} = d_{l-1}, d_{k+2} = d_{l-2},$ etc., we have that $r = e'_l + d_{l-1} + \cdots + d_{l-q}$, and $b' > i$. The converse also holds following a similar argument.
In view of this observation, canceling out the common behavior of $i$ with respect to $b$ and $b'$, the letter $i$ appears one extra time in the minimal cluster of $v$. Since $u\sim_{ss} v$, this is a contradiction. 

Now suppose that $r\geq |v|$. Consider again the embedding $E=\{1, r+2, m\}$. Let $b$ denote the letter that appears right above the letter $i$ of the middle word in the pre-cluster of $u$ and let $b'$ denote the letter that appears right below the letter $i$ of the first word in the pre-cluster of $v$. The claim $b> i \Leftrightarrow b'>i$ follows by symmetry, as in the previous case. Using similar arguments, the letter $i$ appears one extra time in the pre-cluster of $v$, a contradiction.

Let us now suppose that no letter greater than $i$ precedes $i$ to the left or right. Without loss of generality, we have the following configurations
\[ \fbox{i} \ \underbrace{\cdots}_{d_0-1} \ * \ \underbrace{\cdots}_{d_1-1} \ * \ \cdots \ * \ \underbrace{\cdots}_{d_{n-i-1}-1} \ * \quad \textrm{ and } \quad 
* \ \underbrace{\cdots}_{d_1-1} \ * \ \cdots \ * \ \underbrace{\cdots}_{d_{n-i-1}-1} \ * \ \underbrace{\cdots}_{d_0-1} \ \fbox{i}, 
\]
for the corresponding factors of $u$ and $v$, respectively. Since $\Delta_i(s) \neq \Delta_i(t)$, we let $k$ be the smallest index such that $d_0 + d_1 + \cdots + d_{k-1} \neq d_1 + d_2 + \cdots + d_{k-1} + d_k$. It follows that $d_0 = d_1 = \cdots = d_{k-1} =d$, for a suitable positive integer $d$. Then we consider the embedding $E=\{1, 1+kd, 1+kd+\min\{d,d_{k}\} \}$.

Our configurations can be written in the form 
\[ \fbox{i} \ u_1 \ * \ u_2 \ * \ \cdots \ * \ u_{k} \ * \ v \ * \  w   \quad \textrm{ and } \quad w' \ * v' * \ u'_k \ * \ u'_{k-1} \ * \ \cdots \ * \ u'_1 \ \fbox{i}, \]
where $|u_j| = |u'_j| = d$ for $j=1,\ldots, k$, $|v| = |v'|=d_{k}$, and $|w| = |w'|$. Suppose that $d_k < d$. Then, $u_1'$ can be written as $u_1' = w_1' b' v_1'$, where $|v_1'| = d_k$, $b'\in \mathbb{P}$, $b'<i$ and $w_1'\in \mathbb{P}^*$. Therefore, we have the following parts of the pre-clusters for $u$ and $v$, respectively

\[ \begin{array}{cccccccccccccccccc}
       \fbox{i} & u_1 & *   & \cdots & * & u_k & * & v & * & w & & & & & & & &   \\
         & & & & & & \fbox{i} & \cdots  & \cdots   & \cdots & \cdots & \cdots & \cdots & \cdots & \cdots & \cdots & &  \\
         & & & & & & & & \fbox{i}  & u_1 &  *   & \cdots & * & u_k & * & v & * & w 
\end{array} \]

\vspace{0.5cm}

\[ \begin{array}{cccccccccccccccccccc}
       \! \! \! \cdots &\! \! \! * & \! \! \! u_k'   &\! \! * & \cdots & \! \! * & w_1' & b' & v_1' & \fbox{i} & & & & & & & & & &  \\
         & & & & & &        \cdots & \cdots & \cdots &    * & u_k' & * & \cdots & * & w_1' & b' & v_1' & \fbox{i} & &  \\
         & & & & & & & &        \cdots & \cdots & \cdots & \cdots & \cdots   & \cdots & \cdots & * & w_1' & b' & v_1' & \fbox{i}. \\
\end{array} \]
Clearly, the letter $i$ is inherited once in the former minimal cluster, whereas it is inherited twice in the latter one.

The case where $d<d_k$ is dealt in a similar way.

\end{proof}

\begin{example}\label{ex:sswilfclass}
Let $n = 8$ and let $u = 21365874$, $v = 21657843$ and $w = 21478563$. Then set $s = u^{-1} = 21385476$, $t = v^{-1} = 21874356$ and
$p = w^{-1} = 21836745$.

For $i = 7$ down to $2$ the proper suffixes of $s$ are $76$, $476$, $5476$, $85476$, $385476$ and $1385476$. The alphabet sets of these factors are ${\Sigma}_{7}(s) = \{ 6, 7 \}$, ${\Sigma}_{6}(s) = \{4, 6, 7 \}$, ${\Sigma}_{5}(s) = \{4, 5, 6, 7 \}$,
${\Sigma}_{4}(s) = \{4, 5, 6, 7, 8 \}$, ${\Sigma}_{3}(s) = \{3, 4, 5, 6, 7, 8 \}$ and ${\Sigma}_{2}(s) = \{1, 3, 4, 5, 6, 7, 8 \}$.
The corresponding difference vectors are ${\Delta}_{7}(s) = (1)$, ${\Delta}_{6}(s) = (2, 1)$, ${\Delta}_{5}(s) = (1, 1, 1)$,
${\Delta}_{4}(s) = (1, 1, 1, 1)$, ${\Delta}_{3}(s) = (1, 1, 1, 1, 1)$ and ${\Delta}_{2}(s) = (2, 1, 1, 1, 1, 1)$.

The proper suffixes of $t$ are $56$, $356$, $4356$, $74356$, $874356$ and $1874356$. Their alphabet sets are
${\Sigma}_{7}(t) = \{ 5, 6 \}$, ${\Sigma}_{6}(t) = \{3, 5, 6 \}$, ${\Sigma}_{5}(t) = \{3, 4, 5, 6 \}$, ${\Sigma}_{4}(t) = \{3, 4, 5, 6, 7 \}$,
${\Sigma}_{3}(t) = \{3, 4, 5, 6, 7, 8 \}$ and ${\Sigma}_{2}(t) = \{1, 3, 4, 5, 6, 7, 8 \}$.
It is straightforward to check that the difference vectors of $t$ are identical to the corresponding ones for $s$, and consequently we obtain
that $u {\sim}_{ss} v$. 

On the other hand, the proper suffix $745$ of $p$ has alphabet set equal to ${\Sigma}_{6}(p) = \{ 4, 5, 7 \}$ and the corresponding vector of differences is ${\Delta}_{6}(p) = (1, 2) \neq (2, 1) = {\Delta}_{6}(s)$. Therefore $w \nsim_{ss} u$.

Let us now calculate the class $[u]_{ss}$. All possible permutations that satisfy the sequence of differences that correspond to $s=u^{-1}$ are the following: $21385476, 21385467,$  $21835467, 21835476, 21346578, 21346587, 21874356, 21874365$. Taking the inverse of each such permutation, we obtain $[u]_{ss}$ as the class \[ \{21365874, 21365784, 21465783, 21465873, 21346578, 21346587, 21657843, 21658743 \}. \]
Observe that $\big|[u]_+\big|=8 = 2^3$. This is not a coincidence. In the next section, using a binary tree representation for $[u]_+$ we will prove that the cardinality of each super-strong Wilf equivalence class is a power of $2$.

\end{example}

We conclude this section with an application of Theorem \ref{seqofdiff} that demonstrates its feasibility and gives an immediate characterization of the words $w$ for which $w\sim_{ss} \widetilde{w}$. 

Two important super-strong Wilf equivalence classes are the classes \[ \mathcal{I}_n = [123\ldots n]_{ss} \qquad \textrm{and} \qquad \mathcal{M}_n = [12\ldots (n-3) (n-1) (n-2) n]_{ss},\]
for $n\geq 1$ and $n\geq 3$, respectively. It is easy to check that for $u\in \mathcal{I}_n$, $\Delta_i(u^{-1})$ is the $(n-i)$-tuple with all entries equal to $1$, for $i=1,\ldots,n-1$, whereas for $u\in \mathcal{M}_n$, $\Delta_{n-1}(u^{-1}) = (2)$ and $\Delta_i(u^{-1})$ is the $(n-i)$-tuple with all entries equal to $1$, for $i=1,\ldots,n-2$. Observe that in both cases, the vectors of consecutive differences are always palindromic.

\begin{theorem}
Let $w\in\mathcal{S}_n$. Then $w\sim_{ss} \widetilde{w}$ if and only if either $w\in \mathcal{I}_n$ or $w\in\mathcal{M}_n$.
\end{theorem}

\begin{proof}
By Theorem \ref{seqofdiff} we have that $w\sim_{ss} \widetilde{w}$ if and only if $\Delta_i(w^{-1}) = \Delta_i(\widetilde{w}^{-1})$, for $i=1,\ldots, n-1$. Viewing vectors as words, it is easy to check that $\Delta_i(\widetilde{w}^{-1}) = \widetilde{\Delta_i(w^{-1})}$, hence $w\sim_{ss} \widetilde{w}$ if and only if $\Delta_i(w^{-1})$ is a palindrome.

The above remark immediately implies that $\widetilde{\mathcal{I}_n} = \mathcal{I}_n$ and $\widetilde{\mathcal{M}_n} = \mathcal{M}_n$. For the converse, let $i$ be the largest index such that $\Delta_i(w^{-1}) = (1, \ldots, 1)$. If $i = n-1$, then $w\in \mathcal{I}_n$. On the other hand, if $i<n-1$, then we necessarily get \[\Delta_{i+1}(w^{-1}) = (\underbrace{1, 1, \ldots, 1}_r,2,\underbrace{1, 1, \ldots, 1}_r).\] If $r=0$, then $i=n-2$ and clearly $w\in\mathcal{M}_n$. If $r>0$, then all possible choices for $\Delta_{i+2}(w^{-1})$ correspond to non-palindromic vectors.
\end{proof}

\noindent \emph{Remark}. The only words $w$ that do not begin or end in $1$ and for which we have $w\sim_{ss} \widetilde{w}$ are the words $213$ and $312$ which constitute the class $\mathcal{M}_3$.

\section{Binary Tree Representation}\label{Sec:BinaryTree}

The binary tree representation that will be presented here corresponds to the reconstruction of a word $u$ and its cross equivalent words, using the sets $i^+(u)$, for $i = 1, \ldots, n-1$. For this representation we need to define the following sets of partly-filled words of length $n$, on the alphabet $A = \{ 1, 2, \ldots, n, * \}$, where $*$ is an extra character. For $i \in [0,n]$ we set 
\[ S^n_i = \{ x \in A^n \ : \ |x|_j = 1 \text{ for } 1\leq j \leq i \ \text{ and } |x|_* = n-i\}. \]
Observe that for $i=0$ we have $S^n_0 = \{*^n\}$ and for $i=n$ we obtain $S_n^n = \mathcal{S}_n$.

Fix a word $u = u_1 u_2 \ldots u_j \ldots u_n \in \mathcal{S}_n$. We denote by $T^n(u)$ the ordered rooted tree whose leaves constitute the cross equivalence class of $u$. This tree is defined in the following way:
\begin{itemize}
\item The root of the tree is $*^n \in S^n_0$.
\item The elements at the $i$-th level constitute the set \[L_i^n(u) = \{ x\in S_i^n \ : \ d_x(i,*) = i^+(u)\}.\]
\item The word $y=y_1 y_2 \ldots y_n \in S_{i+1}^n$ is a child of the word $x = x_1 x_2 \ldots x_n \in S_i^n$ if and only if for all $j\in [1,i]$ there exists an index $k$ such that $x_k = y_k = j$. In other words, the letters $1, 2, \ldots, i$ appear in the same positions in both $x$ and $y$.
\item The order for the children of the same vertex is defined as follows. If $y = y_1 y_2\ldots y_n$ and $y'=y_1' y_2' \ldots y_n'$ are two children of $x$, then $y$ is \emph{to the left} of $y'$ when for indices $k$ and $l$ such that $y_k = y'_{l} = i+1$, we have $k< l$, otherwise $y$ is \emph{to the right} of $y'$.
\end{itemize}
Note that $L_i^n(u) \neq \varnothing$ for $i\in [0,n]$, since it contains a word $u^{(i)} = u_{i1} u_{i2} \ldots u_{in}$ such that $u_{ij} = u_j$ if $u_j \leq i$ and $u_{ij} = *$ if $u_j> i$. Obviously, for this word the condition $d_u(i,*) = i^+(u)$ holds. Observe that in this notation we have $u^{(0)} = *^n$, $u^{(n)} = u$ and $L_n^n(u) = [u]_+$.

\begin{proposition}\label{prop:tree1or2chil}
The tree $T^n(u)$ is a binary tree, where at each level the number of children is the same throughout all nodes and is either equal to $1$ or $2$.
\end{proposition}

\begin{proof}
Suppose $x\in L_i^n$, where for brevity $L_i^n = L_i^n(u)$. Let $f(x)$ be the factor of $x$ whose first and last letter is respectively the first and last $*$ that appear in $x$. Let us replace each $j\in alph(f(x))$, where $1\leq j \leq i$, with the character $\circ$. In this way, we obtain a configuration word $c(x)$ on the two-lettered alphabet $\{*, \circ\}$ of length $|f(x)|$. Note that this configuration also appears in \eqref{configuration}. Our induction hypothesis is that at each level $i$, one of the following holds:

\begin{enumerate}
\item $|\{ c(x) \ | \ x\in L_i^n\}| = 2$ and for any fixed $x\in L_i^n$ it holds that $L_i^n = \{ c(x), \widetilde{c(x)}\}$. In this case, we have exactly one child for each parent $x\in L_i^n$.
\item $|\{ c(x) \ | \ x\in L_i^n\}| = 1$ and for all $x\in L_i^n$ it holds that $c(x) = \widetilde{c(x)}$.
\begin{enumerate}
\item If $ |c(x)|$ is odd, with the character in the middle position equal to $*$ and $\frac{|c(x)|-1}{2} \in (i+1)^+(u)$, we have exactly one child for each parent $x\in L_i^n$.
\item In all other cases, we have exactly two children for each parent $x\in L_i^n$.
\end{enumerate}
\end{enumerate}

For the first step of this procedure, there are three different cases according to the set $1^+(u)$.
\begin{itemize}
\item {\bf Case 1.} $1^+(u) = \{1, 2, \ldots, n-1\}$ \\
In this case, letter $1$ is placed either in position $1$ or in position $n$. Then, we immediately get $L_1^n = \{1\ *^{n-1}, *^{n-1}\ 1\}$ and $c(x) = *^{n-1}$ for both $x\in L_1^n$. 
\item {\bf Case 2.} $1^+(u) = \{ 1, 1, 2, 2, \ldots, \frac{n-1}{2}, \frac{n-1}{2}\}$ (This case holds only for $n$ odd.) \\
The letter $1$ is placed in the middle position $(n+1)/2$. Here we have only one choice for inserting the letter $1$, namely $L_1^n = \{*^{(n-1)/2}\ 1\ *^{(n-1)/2}\}$. In this case, $c(x) = *^{(n-1)/2} \ \circ \ *^{(n-1)/2}$.
\item {\bf Case 3.} $1^+(u) = \{1, 1, 2, 2, \ldots, l, l, l+1, l+2, \ldots, k\}$, where $1\leq l < k$ and $k+l = n-1$. The letter $1$ is neither in positions $1$ or $n$, nor in the middle position. Here there are two choices for each position, namely $l+1$ or $k+1$. In this case, $L_1^n = \{*^{l}\ 1\ *^{k}, *^{k}\ 1\ *^{l}\}$. Thus, $c(*^{l}\ 1\ *^{k}) = *^{l}\ \circ \ *^{k} = \widetilde{*^{k}\ \circ \ *^{l}} = \widetilde{c(*^{k}\ 1\ *^{l})}$.
\end{itemize}
In all three cases, our desired results hold after inserting $1$. 

Suppose that the induction hypothesis holds for the level $i$. Define $k = \max ((i+1)^+(u))$. The letter $i+1$ will be inserted either in position $k+1$ or in position $|f(x)|-k$ of the word $f(x)$.

In Case 1, we cannot have both choices for placing the letter $1$, because this would imply symmetry, i.e. $\widetilde{c(x)} = c(x)$, a contradiction. Consider $x, x' \in L_i^n$ such that $c(x') \neq c_i(x)$ but $c(x')= \widetilde{c(x)}$. Let $y, y'$ denote their children, respectively. If $i+1$ is inserted in position $k+1$ of $f(x)$, then  it will necessarily be symmetrically inserted in position $|f(x)|-k$ of $f(x')$ and this yields $c(y') = \widetilde{c(y)}$. 

In Case 2 (a), for every word $x \in L_i^n$, its corresponding configuration $c(x)$ will be written as $c(x) = z * z$, for a suitable word $z$. Note that in this case, $\displaystyle k = \frac{|c(x)| + 1}{2}$. Clearly, for the unique child $y$ of $x$, its corresponding word $c(y)$ will be written as $c(y) = z \circ z$.

In Case 2 (b), we have two children for every parent $x$, namely $y, y'$. Suppose, without loss of generality, that $y$ is created by inserting $i+1$ in position $k+1$ of $f(x)$. Then, by symmetry, $y'$ is created by inserting $i+1$ in position $|f(x)|-k$ of $f(x)$. Clearly, we would have that $c(y') = \widetilde{c(y)}$. 

\end{proof}

\begin{corollary}\label{cor:CardCrossEquiv}
The number of permutations in a cross equivalence class is a power of 2.
\end{corollary}

\begin{proof}
The result follows from the equality $[u]_+ = L_n^n(u)$.
\end{proof}

The question now is how cross equivalence classes are partitioned into super-strong Wilf equivalence classes. In order to deal with this, we define a labeling on the vertices of $T^n(u)$ that have two children, distinguishing between ``good'' ones, which preserve symmetry (labeled $0$), and ``bad'' ones which destroy symmetry (labeled $1$).

\begin{definition}\label{def:labeling}
A vertex $x \in T^n(u)$ that has two children $y$ and $y'$ is labeled $0$ if $c(y)=c(y')$, and $1$ otherwise.
\end{definition}

It follows from the proof of Proposition \ref{prop:tree1or2chil} that vertices with the same level have the same labeling.

\begin{theorem}
Let $u, v \in \mathcal{S}_n$. Suppose that $u\sim_+ v$. Then $u\sim_{ss} v$ if and only if one can get from $u$ to $v$ in the cross equivalence tree $T_n(u)$ by following a path that avoids switching direction (from left to right or vice-versa) on vertices at the same level which are labeled 1.
\end{theorem}

\begin{proof}
Consider the unique path $u^{(0)} = *^n \rightarrow u^{(1)} \rightarrow \cdots \rightarrow u^{(i)} \rightarrow \cdots \rightarrow u^{(n)}= u$ from the root of $T^n(u)$ to the leaf $u$. Let $f_i(u)$ and $c_i(u)$ be respectively the factor $f(u^{(i)})$ of $u^{(i)}$ and its configuration $c(u^{(i)})$. 

Suppose $c_i(u) = c^{(i)}_1 c^{(i)}_2 \ldots c^{(i)}_{|f_i(u)|}$. Define $\Sigma_i(u) = \{j : c^{(i)}_j = * \}$ and observe that if we arrange it in ascending order we obtain \[ \Sigma_i(u) = \{j_1 < j_2 < \cdots < j_{n-i} \}. \]
Recall that the sets $\Sigma_i(u^{-1}) = \Sigma_i(s)$ of Definition \ref{def:SeqDiff} represent the positions of the $n-i$ letters in $u$ that are greater than $i$. It is crucial to observe that they also represent the positions of $*$ in $u^{(i)}$. Since $u^{(i)}$ can be written in the form $u^{(i)} = p f_i(u) q$ for suitable words $p,q \in [1,i-1]^*$, this observation yields \begin{equation}\label{eq:rescaling} \Sigma_{i+1}(s) = |p| + \Sigma_i(u).\end{equation} 
This change of index is due to the following fact. In both cases, we consider distances between letters which are greater than $i$. These correspond precisely to the sets $\Sigma_{i+1}(s)$ and $\Sigma_i(u)$ that appear to the left and right hand side of (\ref{eq:rescaling}). It follows that 
\begin{equation}\label{eq:diffc(u)} s^{(i+1)}_{i+l} - s^{(i+1)}_{i+l-1} = j_{l+1} - j_l, \quad l\in [1,n-i-1]. \end{equation}

Let $v\in [u]_{ss}$. By Theorem \ref{seqofdiff}, this is equivalent to $\Delta_{i+1}(u^{-1}) = \Delta_{i+1}(v^{-1})$ for $i\in [1,n-2]$. In view of equation (\ref{eq:diffc(u)}), this is equivalent to $\Sigma_{i}(u) = \Sigma_{i}(v)$ for $i\in [1,n-2]$ or, in other words, $c_i(u) = c_i(v)$. Going back to Definition \ref{def:labeling}, which provides a labeling on $T^n(u)$, the result follows.
\end{proof}

\begin{corollary}
Let $u\in \mathcal{S}_n$ and let $k, l$ be the number of levels in $T^n(u)$ labeled 0 and 1, respectively. Then:
\begin{itemize}
\item The number of words in each super-strong Wilf equivalence class in $T^n(u)$ is equal to $2^k$.
\item The class $[u]_+$ is partioned into $2^l$ distinct super-strong Wilf equivalence classes.
\end{itemize}
\end{corollary}

\begin{proof}
In order to find the words $v$ that are super-strong Wilf equivalent to $u$, we follow a path in $T_n(u)$ that can change direction (from left to right or vice-versa) only on vertices labeled $0$ at the same level. This provides us with two choices for every such level. This implies the first statement. Now, since $T^n(u)$ has $2^{k+l}$ leaves, the second statement follows.
\end{proof}

\begin{example} \label{ex:tree}Let us construct the tree $T^n(u)$ for $n=8$ and the word $u= 2 1 3 6 5 8 7 4$. First, we find the multisets of distances for the word $u$. These are \\

\begin{tabular}{l l l }
$7^+(u) = \{1\}$, & $6^+(u) = \{2,3\}$, & $5^+(u) = \{1,1,2\}$, \\ $4^+(u) = \{1,2,3,4\}$, &
$3^+(u) = \{ 1, 2, 3, 4, 5 \}$, & \\ $2^+(u) = \{ 2, 3, 4, 5, 6, 7\}$, and & $1^+(u) = \{1, 1, 2, 3, 4, 5, 6 \}$.
\end{tabular} \\

Using the above, we find all words that have the same multisets of distances by placing the corresponding letter at each step and considering all possible choices at each level. This yields the tree $T^n(u)$ shown in page \pageref{fig:tree}. 

The following table traces the path along the vertices of the tree $T^n(u)$ beginning at the root and leading to the leaf $u$. The corresponding configuration words $c_i(u)$ and vectors of differences $\Delta_{i+1}(u^{-1})$ for $u^{-1}$ are also given at each step, for $i\in[0,8]$.

\begin{center}
\begin{tabular}{l | c c c c c c c c | c | c}
$i$ & & & & $u^{(i)}$& & & & & $c_i(u)$ & $\Delta_{i+1}(u^{-1})$ \\
\hline
0 & $*$ & $*$ & $*$ & $*$ & $*$ & $*$ & $*$ & $*$ & $ * \ * \ * \ * \ * \ * \ * \ *$     & $(1,1,1,1,1,1,1)$ \\
1 & $*$ & $1$ & $*$ & $*$ & $*$ & $*$ & $*$ & $*$ & $ * \ \circ \ * \ * \ * \ * \ * \ *$ & $(2,1,1,1,1,1)$ \\
2 & $2$ & $1$ & $*$ & $*$ & $*$ & $*$ & $*$ & $*$ & $ * \ * \ * \ * \ * \ *$             & $(1,1,1,1,1)$ \\
3 & $2$ & $1$ & $3$ & $*$ & $*$ & $*$ & $*$ & $*$ & $ * \ * \ * \ * \ * $                & $(1,1,1,1)$ \\
4 & $2$ & $1$ & $3$ & $*$ & $*$ & $*$ & $*$ & $4$ & $ * \ * \ * \ *$                     & $(1,1,1)$\\
5 & $2$ & $1$ & $3$ & $*$ & $5$ & $*$ & $*$ & $4$ & $ * \ \circ \ * \ *$                 & $(2,1)$\\
6 & $2$ & $1$ & $3$ & $6$ & $5$ & $*$ & $*$ & $4$ & $ * \ * $                            & $(1)$\\
7 & $2$ & $1$ & $3$ & $6$ & $5$ & $*$ & $7$ & $4$ & $ * $                                & $-$\\
8 & $2$ & $1$ & $3$ & $6$ & $5$ & $8$ & $7$ & $4$ &      $-$                             & $-$\\
\end{tabular}
\end{center}

Let $v = 21347856$. The super-strong Wilf equivalence classes obtained by the tree $T^n(u)$, starting from the class $[u]_{ss}$ and reading the leaves of the tree from left to right, where $u$, $v$ and their reversals are underlined within their classes, are the following:

\vspace{-0.5cm}
\begin{center}
\begin{tabular}{| l | l |}
\hline
 & Class \\
\hline
$u$ & $21346578, 21346587, 21365784, \underline{21365874}, 21465783, 21465873, 21657843, 21658743$ \\
\hline
$v$ & $\underline{21347856}, 21348756, 21378564, 21387564, 21478563, 21487563, 21785643, 21875643$ \\
\hline
$\tilde{u}$ & $34785612, 34875612, 37856412, 38756412, \underline{47856312}, 48756312, 78564312, 87564312$ \\
\hline
$\tilde{v}$ & $34657812, 34658712, 36578412, 36587412, 46578312, 46587312, 65784312, \underline{65874312}$ \\
\hline
\end{tabular}
\end{center}

The above classes are distinguished in the tree (see page \pageref{fig:tree}) as follows: elements of $[u]_{ss}$ in red boxes, elements of $[v]_{ss}$ in blue boxes and their reversals in the corresponding dashed boxes. Finally, $0$ and $1$ labels are shown with green and orange color, respectively.

\end{example}

\section{Conclusion}

Recently the geometric notion of \emph{shift equivalence} was defined and studied in \cite{FGMPX}. In the same paper it was shown that shift equivalence implies strong Wilf and therefore Wilf equivalence. We would like to know if there are connections amongst Wilf, cross, shift, and super-strong Wilf equivalence classes. Suppose that $u\nsim_{ss} \widetilde{u}$. By Lemma \ref{lem:Kitaev} we get $[u]_{ss} \cup [\widetilde{u}]_{ss}\subseteq [u]$ and since $u\sim_+ \widetilde{u}$ we also obtain $[u]_{ss} \cup [\widetilde{u}]_{ss}\subseteq [u]_+$. Is there a specific relationship between $[u]_+$ and $[u]$? For $n\leq 5$, we have verified that $[u]_{ss} \cup [\widetilde{u}]_{ss} = [u] = [u]_+$. On the other hand, in Example \ref{ex:tree}, we found a word $u$ such that $[u]_{ss} \cup [\widetilde{u}]_{ss}\neq [u]_+$. This led us to the question whether $[u]_{ss} \cup [\widetilde{u}]_{ss} = [u]$. It turns out that the answer is negative. Let $u=2 3 4 1 5 6$ and $v = 2 5 6 1 4 3$. Then, as a by-product from \cite[Section 5]{FGMPX}, $u\sim_s v$, and therefore $u\sim v$. On the other hand $v\nsim_{ss} u$ and $v\nsim_{ss} \widetilde{u}$. 

\begin{problem}
Is it true that $[u]\subseteq [u]_+$ or do there exist words $v\in [u]_+$ such that $v\nsim u$?
\end{problem}

\begin{problem}
Enumerate all cross equivalence and super-strong Wilf equivalence classes for a given $n\in \mathbb{N}$.  
\end{problem}

\newpage 
\voffset=-2.5cm
\hoffset=-2.5cm
\thispagestyle{empty}

\begin{turn}{90}
\begin{tikzpicture}[scale=1.64]
\label{fig:tree}
{\tiny

\node at (0,9.5) {* \, * \, * \, * \, * \, * \, * \, *};

\node at (0,9) {\color{orange}{$1$}};
\draw [orange] (0,9) circle [radius=2pt];

\draw[->](0,9.2) -- (-4,7.7);
\node at (-4,7.4) {* \, 1 \, * \, * \, * \, * \, * \, *};
\draw[->](0,9.2) -- (4,7.7);
\node at (4,7.4) {* \, * \, * \, * \, * \, * \, 1 \, *};

\draw[->](-4,7.2) -- (-4,6.2);
\node at (-4,6) {2 \, 1 \, * \, * \, * \, * \, * \, *};
\node at (-4,5.5) {\color{green}{$0$}};
\draw [green] (-4,5.5) circle [radius=2pt];
\draw[->](4,7.2) -- (4,6.2);
\node at (4,6) {* \, * \, * \, * \, * \, * \,  1 \, 2};
\node at (4,5.5) {\color{green}{$0$}};
\draw [green] (4,5.5) circle [radius=2pt];

\draw[->](-4,5.7) -- (-6,4.3);
\node at (-6,4) {2 \, 1 \, 3 \, * \, * \, * \, * \, *};
\node at (-6,3.5) { \color{green}{$0$}};
\draw [green] (-6,3.5) circle [radius=2pt];
\draw[->](-4,5.7) -- (-2,4.3);
\node at (-2,4) {2 \, 1 \, * \, * \, * \, * \, * \, 3};
\node at (-2,3.5) {\color{green}{$0$}};
\draw [green] (-2,3.5) circle [radius=2pt];
\draw[->](4,5.7) -- (6,4.3);
\node at (6,4) {* \, * \, * \, * \, * \, 3 \, 1 \, 2};
\node at (6,3.5) {\color{green}{$0$}};
\draw [green] (6,3.5) circle [radius=2pt];
\draw[->](4,5.7) -- (2,4.3);
\node at (2,4) {3 \, * \, * \, * \, * \, * \, 1 \, 2};
\node at (2,3.5) {\color{green}{$0$}};
\draw [green] (2,3.5) circle [radius=2pt];

\draw[->](-6,3.7) -- (-7,2.3);
\node at (-7,2) {2 \, 1 \, 3 \, 4 \, * \, * \, * \, *};
\node at (-7,1.5) {\color{orange}{$1$}};
\draw [orange] (-7,1.5) circle [radius=2pt];
\draw[->](-6,3.7) -- (-5,2.3);
\node at (-5,2) {2 \, 1 \, 3 \, * \, * \, * \, * \, 4};
\node at (-5,1.5) {\color{orange}{$1$}};
\draw [orange] (-5,1.5) circle [radius=2pt];
\draw[->](-2,3.7) -- (-3,2.3);
\node at (-3,2) {2 \, 1 \, 4 \, * \, * \, * \, * \, 3};
\node at (-3,1.5) {\color{orange}{$1$}};
\draw [orange] (-3,1.5) circle [radius=2pt];
\draw[->](-2,3.7) -- (-1,2.3);
\node at (-1,1.5) {\color{orange}{$1$}};
\draw [orange] (-1,1.5) circle [radius=2pt];
\node at (-1,2) {2 \, 1 \, * \, * \, * \, * \, 4 \, 3};
\draw[->](6,3.7) -- (7,2.3);
\node at (1,1.5) {\color{orange}{$1$}};
\draw [orange] (1,1.5) circle [radius=2pt];
\node at (1,2) {3 \, 4 \, * \, * \, * \, * \, 1 \, 2};
\draw[->](6,3.7) -- (5,2.3);
\node at (3,1.5) {\color{orange}{$1$}};
\draw [orange] (3,1.5) circle [radius=2pt];
\node at (3,2) {3 \, * \, * \, * \, * \, 4 \, 1 \, 2};
\draw[->](2,3.7) -- (3,2.3);
\node at (5,1.5) {\color{orange}{$1$}};
\draw [orange] (5,1.5) circle [radius=2pt];
\node at (5,2) {4 \, * \, * \, * \, * \, 3 \, 1 \, 2};
\draw[->](2,3.7) -- (1,2.3);
\node at (7,1.5) {\color{orange}{$1$}};
\draw [orange] (7,1.5) circle [radius=2pt];
\node at (7,2) {* \, * \, * \, * \, 4 \, 3 \, 1 \, 2};

\draw[->](-7,1.7) -- (-7.5,0.8);
\node at (-7.5,0.5) {$2 1 3 4 * 5 * *$};
\draw[->](-7,1.7) -- (-6.5,0.8);
\node at (-6.5,0.5) {$2 1 3 4 * * 5 *$};
\draw[->](-5,1.7) -- (-5.5,0.8);
\node at (-5.5,0.5) {$2 1 3 * 5 * * 4$};
\draw[->](-5,1.7) -- (-4.5,0.8);
\node at (-4.5,0.5) {$2 1 3 * * 5 * 4$};
\draw[->](-3,1.7) -- (-3.5,0.8);
\node at (-3.5,0.5) {$2 1 4 * 5 * * 3$};
\draw[->](-3,1.7) -- (-2.5,0.8);
\node at (-2.5,0.5) {$2 1 4 * * 5 * 3$};
\draw[->](-1,1.7) -- (-1.5,0.8);
\node at (-1.5,0.5) {$2 1 * 5 * * 4 3$};
\draw[->](-1,1.7) -- (-0.5,0.8);
\node at (-0.5,0.5) {$2 1 * * 5 * 4 3$};
\draw[->](1,1.7) -- (0.5,0.8);
\node at (0.5,0.5) {$3 4 * 5 * * 1 2$};
\draw[->](1,1.7) -- (1.5,0.8);
\node at (1.5,0.5) {$3 4 * * 5 * 1 2$};
\draw[->](3,1.7) -- (2.5,0.8);
\node at (2.5,0.5) {$3 * 5 * * 4 1 2$};
\draw[->](3,1.7) -- (3.5,0.8);
\node at (3.5,0.5) {$3 * * 5 * 4 1 2$};
\draw[->](5,1.7) -- (4.5,0.8);
\node at (4.5,0.5) {$4 * 5 * * 3 1 2$};
\draw[->](5,1.7) -- (5.5,0.8);
\node at (5.5,0.5) {$4 * * 5 * 3 1 2$};
\draw[->](7,1.7) -- (6.5,0.8);
\node at (6.5,0.5) {$* 5 * * 4 3 1 2$};
\draw[->](7,1.7) -- (7.5,0.8);
\node at (7.5,0.5) {$* * 5 * 4 3 1 2$};

\draw[->](-7.5,0.3) -- (-7.5,-0.3);
\node at (-7.5,-1.15) {\color{green}{$0$}};
\draw [green] (-7.5,-1.15) circle [radius=2pt];
\node at (-7.5,-0.5) {$2 1 3 4 6 5 * *$};
\draw[->](-6.5,0.3) -- (-6.5,-0.3);
\node at (-6.5,-1.15) {\color{green}{$0$}};
\draw [green] (-6.5,-1.15) circle [radius=2pt];
\node at (-6.5,-0.5) {$2 1 3 4 * * 5 6$};
\draw[->](-5.5,0.3) -- (-5.5,-0.3);
\node at (-5.5,-1.15) {\color{green}{$0$}};
\draw [green] (-5.5,-1.15) circle [radius=2pt];
\node at (-5.5,-0.5) {$2 1 3 6 5 * * 4$};
\draw[->](-4.5,0.3) -- (-4.5,-0.3);
\node at (-4.5,-1.15) {\color{green}{$0$}};
\draw [green] (-4.5,-1.15) circle [radius=2pt];
\node at (-4.5,-0.5) {$2 1 3 * * 5 6 4$};
\draw[->](-3.5,0.3) -- (-3.5,-0.3);
\node at (-3.5,-1.15) {\color{green}{$0$}};
\draw [green] (-3.5,-1.15) circle [radius=2pt];
\node at (-3.5,-0.5) {$2 1 4 6 5 * * 3$};
\draw[->](-2.5,0.3) -- (-2.5,-0.3);
\node at (-2.5,-1.15) {\color{green}{$0$}};
\draw [green] (-2.5,-1.15) circle [radius=2pt];
\node at (-2.5,-0.5) {$2 1 4 * * 5 6 3$};
\draw[->](-1.5,0.3) -- (-1.5,-0.3);
\node at (-1.5,-1.15) {\color{green}{$0$}};
\draw [green] (-1.5,-1.15) circle [radius=2pt];
\node at (-1.5,-0.5) {$2 1 6 5 * * 4 3$};
\draw[->](-0.5,0.3) -- (-0.5,-0.3);
\node at (-0.5,-1.15) {\color{green}{$0$}};
\draw [green] (-0.5,-1.15) circle [radius=2pt];
\node at (-0.5,-0.5) {$2 1 * * 5 6 4 3$};
\draw[->](0.5,0.3) -- (0.5,-0.3);
\node at (0.5,-1.15) {\color{green}{$0$}};
\draw [green] (0.5,-1.15) circle [radius=2pt];
\node at (0.5,-0.5) {$3 4 6 5 * * 1 2$};
\draw[->](1.5,0.3) -- (1.5,-0.3);
\node at (1.5,-1.15) {\color{green}{$0$}};
\draw [green] (1.5,-1.15) circle [radius=2pt];
\node at (1.5,-0.5) {$3 4 * * 5 6 1 2$};
\draw[->](2.5,0.3) -- (2.5,-0.3);
\node at (2.5,-1.15) {\color{green}{$0$}};
\draw [green] (2.5,-1.15) circle [radius=2pt];
\node at (2.5,-0.5) {$3 6 5 * * 4 1 2$};
\draw[->](3.5,0.3) -- (3.5,-0.3);
\node at (3.5,-1.15) {\color{green}{$0$}};
\draw [green] (3.5,-1.15) circle [radius=2pt];
\node at (3.5,-0.5) {$3 * * 5 6 4 1 2$};
\draw[->](4.5,0.3) -- (4.5,-0.3);
\node at (4.5,-1.15) {\color{green}{$0$}};
\draw [green] (4.5,-1.15) circle [radius=2pt];
\node at (4.5,-0.5) {$4 6 5 * * 3 1 2$};
\draw[->](5.5,0.3) -- (5.5,-0.3);
\node at (5.5,-1.15) {\color{green}{$0$}};
\draw [green] (5.5,-1.15) circle [radius=2pt];
\node at (5.5,-0.5) {$4 * * 5 6 3 1 2$};
\draw[->](6.5,0.3) -- (6.5,-0.3);
\node at (6.5,-1.15) {\color{green}{$0$}};
\draw [green] (6.5,-1.15) circle [radius=2pt];
\node at (6.5,-0.5) {$6 5 * * 4 3 1 2$};
\draw[->](7.5,0.3) -- (7.5,-0.3);
\node at (7.5,-1.15) {\color{green}{$0$}};
\draw [green] (7.5,-1.15) circle [radius=2pt];
\node at (7.5,-0.5) {$* * 5 6 4 3 1 2$};

\draw[->](-7.5,-0.8) -- (-7.75,-1.7);
\node at (-7.75,-2) {\cfbox{red}{$21346578$}};
\draw[->](-7.5,-0.8) -- (-7.25,-1.7);
\node at (-7.25,-1.8) {\cfbox{red}{$21346587$}};
\draw[->](-6.5,-0.8) -- (-6.75,-1.7);
\node at (-6.75,-2) {\cfbox{blue}{$21347856$}};
\draw[->](-6.5,-0.8) -- (-6.25,-1.7);
\node at (-6.25,-1.8) {\cfbox{blue}{$21348756$}};
\draw[->](-5.5,-0.8) -- (-5.75,-1.7);
\node at (-5.75,-2) {\cfbox{red}{$2 1 3 6 5 7 8 4$}};
\draw[->](-5.5,-0.8) -- (-5.25,-1.7);
\node at (-5.25,-1.8) {\cfbox{red}{$2 1 3 6 5 8 7 4$}};
\draw[->](-4.5,-0.8) -- (-4.75,-1.7);
\node at (-4.75,-2) {\cfbox{blue}{$2 1 3 7 8 5 6 4$}};
\draw[->](-4.5,-0.8) -- (-4.25,-1.7);
\node at (-4.25,-1.8) {\cfbox{blue}{$2 1 3 8 7 5 6 4$}};
\draw[->](-3.5,-0.8) -- (-3.75,-1.7);
\node at (-3.75,-2) {\cfbox{red}{$2 1 4 6 5 7 8 3$}};
\draw[->](-3.5,-0.8) -- (-3.25,-1.7);
\node at (-3.25,-1.8) {\cfbox{red}{$2 1 4 6 5 8 7 3$}};
\draw[->](-2.5,-0.8) -- (-2.75,-1.7);
\node at (-2.75,-2) {\cfbox{blue}{$2 1 4 7 8 5 6 3$}};
\draw[->](-2.5,-0.8) -- (-2.25,-1.7);
\node at (-2.25,-1.8) {\cfbox{blue}{$2 1 4 8 7 5 6 3$}};
\draw[->](-1.5,-0.8) -- (-1.75,-1.7);
\node at (-1.75,-2) {\cfbox{red}{$2 1 6 5 7 8 4 3$}};
\draw[->](-1.5,-0.8) -- (-1.25,-1.7);
\node at (-1.25,-1.8) {\cfbox{red}{$2 1 6 5 8 7 4 3$}};
\draw[->](-0.5,-0.8) -- (-0.75,-1.7);
\node at (-0.75,-2) {\cfbox{blue}{$2 1 7 8 5 6 4 3$}};
\draw[->](-0.5,-0.8) -- (-0.25,-1.7);
\node at (-0.25,-1.8) {\cfbox{blue}{$2 1 8 7 5 6 4 3$}};
\draw[->](0.5,-0.8) -- (0.25,-1.7);
\draw[dashed, blue] (-0.1,-1.9) -- (-0.1, -2.1) -- (0.6, -2.1) -- (0.6,-1.9) -- (-0.1,-1.9); 
\node at (0.25,-2) {$3 4 6 5 7 8 1 2$};
\draw[->](0.5,-0.8) -- (0.75,-1.7);
\draw[dashed, blue] (0.4,-1.7) -- (0.4, -1.9) -- (1.1, -1.9) -- (1.1,-1.7) -- (0.4,-1.7); 
\node at (0.75,-1.8) {$3 4 6 5 8 7 1 2$};
\draw[->](1.5,-0.8) -- (1.25,-1.7);
\draw[dashed, red] (0.9,-1.9) -- (0.9, -2.1) -- (1.6, -2.1) -- (1.6,-1.9) -- (0.9,-1.9); 
\node at (1.25,-2) {$3 4 7 8 5 6 1 2$};
\draw[->](1.5,-0.8) -- (1.75,-1.7);
\draw[dashed, red] (1.4,-1.7) -- (1.4, -1.9) -- (2.1, -1.9) -- (2.1,-1.7) -- (1.4,-1.7); 
\node at (1.75,-1.8) {$3 4 8 7 5 6 1 2$};
\draw[->](2.5,-0.8) -- (2.25,-1.7);
\draw[dashed, blue] (1.9,-1.9) -- (1.9, -2.1) -- (2.6, -2.1) -- (2.6,-1.9) -- (1.9,-1.9); 
\node at (2.25,-2) {$3 6 5 7 8 4 1 2$};
\draw[->](2.5,-0.8) -- (2.75,-1.7);
\draw[dashed, blue] (2.4,-1.7) -- (2.4, -1.9) -- (3.1, -1.9) -- (3.1,-1.7) -- (2.4,-1.7); 
\node at (2.75,-1.8) {$3 6 5 8 7 4 1 2$};
\draw[->](3.5,-0.8) -- (3.25,-1.7);
\draw[dashed, red] (2.9,-1.9) -- (2.9, -2.1) -- (3.6, -2.1) -- (3.6,-1.9) -- (2.9,-1.9); 
\node at (3.25,-2) {$3 7 8 5 6 4 1 2$};
\draw[->](3.5,-0.8) -- (3.75,-1.7);
\draw[dashed, red] (3.4,-1.7) -- (3.4, -1.9) -- (4.1, -1.9) -- (4.1,-1.7) -- (3.4,-1.7); 
\node at (3.75,-1.8) {$3 8 7 5 6 4 1 2$};
\draw[->](4.5,-0.8) -- (4.25,-1.7);
\draw[dashed, blue] (3.9,-1.9) -- (3.9, -2.1) -- (4.6, -2.1) -- (4.6,-1.9) -- (3.9,-1.9); 
\node at (4.25,-2) {$4 6 5 7 8 3 1 2$};
\draw[->](4.5,-0.8) -- (4.75,-1.7);
\draw[dashed, blue] (4.4,-1.7) -- (4.4, -1.9) -- (5.1, -1.9) -- (5.1,-1.7) -- (4.4,-1.7); 
\node at (4.75,-1.8) {$4 6 5 8 7 3 1 2$};
\draw[->](5.5,-0.8) -- (5.25,-1.7);
\draw[dashed, red] (4.9,-1.9) -- (4.9, -2.1) -- (5.6, -2.1) -- (5.6,-1.9) -- (4.9,-1.9); 
\node at (5.25,-2) {$4 7 8 5 6 3 1 2$};
\draw[->](5.5,-0.8) -- (5.75,-1.7);
\draw[dashed, red] (5.4,-1.7) -- (5.4, -1.9) -- (6.1, -1.9) -- (6.1,-1.7) -- (5.4,-1.7); 
\node at (5.75,-1.8) {$4 8 7 5 6 3 1 2$};
\draw[->](6.5,-0.8) -- (6.25,-1.7);
\draw[dashed, blue] (5.9,-1.9) -- (5.9, -2.1) -- (6.6, -2.1) -- (6.6,-1.9) -- (5.9,-1.9); 
\node at (6.25,-2) {$6 5 7 8 4 3 1 2$};
\draw[->](6.5,-0.8) -- (6.75,-1.7);
\draw[dashed, blue] (6.4,-1.7) -- (6.4, -1.9) -- (7.1, -1.9) -- (7.1,-1.7) -- (6.4,-1.7); 
\node at (6.75,-1.8) {$6 5 8 7 4 3 1 2$};
\draw[->](7.5,-0.8) -- (7.25,-1.7);
\draw[dashed, red] (6.9,-1.9) -- (6.9, -2.1) -- (7.6, -2.1) -- (7.6,-1.9) -- (6.9,-1.9); 
\node at (7.25,-2) {$7 8 5 6 4 3 1 2$};
\draw[->](7.5,-0.8) -- (7.75,-1.7);
\draw[dashed, red] (7.4,-1.7) -- (7.4, -1.9) -- (8.1, -1.9) -- (8.1,-1.7) -- (7.4,-1.7); 
\node at (7.75,-1.8) {$8 7 5 6 4 3 1 2$};

}
\end{tikzpicture}
\end{turn}

\newpage
\voffset=0cm
\hoffset=0cm


\begin{thebibliography}{99}
\bibitem{Elizalde} S.~Elizalde, {\it A survey of consecutive patterns in permutations}, in Recent Trends in Combinatorics,
IMA Volume in Mathematics and its Applications, Springer, 2016, pp. 601--618.
\bibitem{FGMPX} J.~Fidler, D.~Glasscock, B.~Miceli, J.~Pantone, and M.~Xu, {\it Shift equivalence in the generalized factor order},
arXiv:1612.09003 [math.CO].
\bibitem{GouldenJackson} I. P. Goulden, D. M. Jackson, {\it Combinatorial Enumeration}, A Wiley-Interscience Publication, John Wiley \& Sons Inc., New York, 1983.
\bibitem{Kitaevbook} S. Kitaev, {\it Patterns in Permutations and Words}, Monographs in Theoretical Computer Science, EATCS Series, Springer-Verlag, 2011.
\bibitem{Kitaev} S. Kitaev, J. Liese, J. Remmel, B. E. Sagan, {\it Rationality, irrationality and Wilf equivalence in generalized factor order}, The Electronic Journal of Combinatorics 16(2), 2009.
\bibitem{Pantone} J. Pantone, V. Vatter, {\it On the Rearrangement Conjecture for generalized factor order over $\mathbb{P}$}, in 26th International Conference on Formal Power Series and Algebraic Combinatorics (FPSAC 2014). Discrete Math. Theor. Comput. Sci. Proc., AT. Assoc. Discrete Math. Theor. Comput. Sci., Nancy, 2014, pp. 217–-228.
\end{thebibliography}
\end{document}